\documentclass[reqno,11pt]{article} 
\usepackage[left=2.5cm,right=2.5cm,top=3cm,bottom=3cm,a4paper]{geometry}
\usepackage{amsmath, amssymb}
\usepackage{amsthm, amscd} 
\usepackage[all,cmtip]{xy}
\usepackage{float}
\usepackage{caption}
\usepackage{color}

\newcommand{\Mod}[1]{\ (\textup{mod}\ #1)}
\theoremstyle{plain} 
\newtheorem{theorem}{\indent\sc Theorem}[section]
\newtheorem{lemma}[theorem]{\indent\sc Lemma}

\newtheorem{proposition}[theorem]{\indent\sc Proposition}

\theoremstyle{definition} 
\newtheorem{definition}[theorem]{\indent\sc Definition}

\newtheorem{remark}[theorem]{\indent\sc Remark}
\newtheorem{example}[theorem]{\indent\sc Example}

\makeatletter
\def\address#1#2{\begingroup
\noindent\parbox[t]{7.8cm}{%
\small{\scshape\ignorespaces#1}\par\vskip1ex
\noindent\small{\itshape E-mail address}%
\/: #2\par\vskip4ex}\hfill%
\endgroup}%
\makeatother
\title{On some extension of Gauss' work and applications}
\author{
\textsc{Ho Yun Jung, Ja Kyung Koo and Dong Hwa Shin} 
}
\date{} 
\begin{document}

\allowdisplaybreaks

\maketitle

 \footnote{
\thanks{
The first (corresponding) author was supported by the National Research Foundation of Korea (NRF) grant funded by the Korea Government (MIST) (2016R1A5A1008055 and
2017R1C1B2010652).
The third author was supported by the National Research Foundation of Korea (NRF) grant funded by the Korea
government (MSIT) (2017R1A2B1006578), and by Hankuk University of Foreign Studies Research Fund of 2018.}
}

\begin{abstract}
Let $K$ be an imaginary quadratic field of discriminant $d_K$ with ring of integers $\mathcal{O}_K$, and let $\tau_K$ be an element of the complex upper half-plane so that $\mathcal{O}_K=[\tau_K,\,1]$. For a positive integer $N$, let $\mathcal{Q}_N(d_K)$ be the set of primitive positive definite binary quadratic forms of discriminant $d_K$ with leading coefficients relatively prime to $N$. Then, with any congruence subgroup $\Gamma$ of $\mathrm{SL}_2(\mathbb{Z})$ one can define an equivalence relation $\sim_\Gamma$ on
$\mathcal{Q}_N(d_K)$. Let $\mathcal{F}_{\Gamma,\,\mathbb{Q}}$ denote the field of meromorphic modular functions for $\Gamma$ with rational Fourier coefficients.
We show that the set of equivalence classes $\mathcal{Q}_N(d_K)/\sim_\Gamma$ can
be equipped with a group structure isomorphic to
$\mathrm{Gal}(K\mathcal{F}_{\Gamma,\,\mathbb{Q}}(\tau_K)/K)$ for some $\Gamma$, which generalizes further into the classical theory of complex multiplication over ring class fields.
\end{abstract}

\maketitle

\section {Introduction}

For a negative integer $D$ such that
$D\equiv0$ or $1\Mod{4}$, let $\mathcal{Q}(D)$ be the
set of primitive positive definite binary quadratic forms
$Q(x,\,y)=ax^2+bxy+cy^2\in\mathbb{Z}[x,\,y]$ of discriminant $b^2-4ac=D$. The modular group $\mathrm{SL}_2(\mathbb{Z})$
(or $\mathrm{PSL}_2(\mathbb{Z})$)
acts on the set $\mathcal{Q}(D)$ from the right and defines the proper equivalence $\sim$ as
\begin{equation*}
Q\sim Q'\quad\Longleftrightarrow\quad
Q'=Q^\gamma
=Q\left(\gamma\begin{bmatrix}x\\y\end{bmatrix}\right)~
\textrm{for some}~\gamma\in\mathrm{SL}_2(\mathbb{Z}).
\end{equation*}
In his celebrated work Disquisitiones Arithmeticae
of 1801 (\cite{Gauss}), Gauss introduced the beautiful law of
composition of integral binary quadratic forms. And, it seems
that he first understood the set of equivalence classes
$\mathrm{C}(D)=\mathcal{Q}(D)/\sim$ as a group.
However, his original proof of the group structure is long
and complicated to work in practice. After 93 years later
Dirichlet (\cite{Dirichlet}) presented a different approach to
the study of composition and genus theory, which seemed to be
influenced by Legendre. (See \cite[$\S$3]{Cox}.) On the other hand,
in 2004 Bhargava (\cite{Bhargava}) derived a wonderful general
law of composition on $2\times2\times2$ cubes of integers, from which
he was able to obtain Gauss' composition law on binary quadratic forms as a simple special case. Now, in this paper we will make use of Dirichlet's composition law to proceed the arguments.
\par
Given the order $\mathcal{O}$ of discriminant $D$ in the imaginary quadratic field $K=\mathbb{Q}(\sqrt{D})$,
let $I(\mathcal{O})$ be the group of proper fractional $\mathcal{O}$-ideals and $P(\mathcal{O})$ be its subgroup
of nonzero principal $\mathcal{O}$-ideals.
When $Q=ax^2+bxy+cy^2$ is an element of $\mathcal{Q}(D)$,
let $\omega_Q$ be the zero of the quadratic polynomial
$Q(x,\,1)$ in $\mathbb{H}
=\{\tau\in\mathbb{C}~|~\mathrm{Im}(\tau)>0\}$, namely
\begin{equation}\label{wQ}
\omega_Q=\frac{-b+\sqrt{D}}{2a}.
\end{equation}
It is well known that $[\omega_Q,\,1]=\mathbb{Z}\omega_Q+\mathbb{Z}$ is
a proper fractional $\mathcal{O}$-ideal and the form class group $\mathrm{C}(D)$ under the Dirichlet composition is
isomorphic to the $\mathcal{O}$-ideal class group $\mathrm{C}(\mathcal{O})=I(\mathcal{O})/P(\mathcal{O})$
through the isomorphism
\begin{equation}\label{CDCO}
\mathrm{C}(D)\stackrel{\sim}{\rightarrow}\mathrm{C}(\mathcal{O}),\quad
[Q]\mapsto[[\omega_Q,\,1]].
\end{equation}
On the other hand,
if we let $H_\mathcal{O}$ be the ring class field of order $\mathcal{O}$ and $j$ be the elliptic modular
function on lattices in $\mathbb{C}$, then we attain the isomorphism
\begin{equation}\label{COGHK}
\mathrm{C}(\mathcal{O})\stackrel{\sim}{\rightarrow}\mathrm{Gal}(H_\mathcal{O}/K),
\quad
[\mathfrak{a}]\mapsto(j(\mathcal{O})
\mapsto j(\overline{\mathfrak{a}}))
\end{equation}
by the theory of complex multiplication
(\cite[Theorem 11.1 and Corollary 11.37]{Cox} or
\cite[Theorem 5 in Chapter 10]{Lang87}). Thus, composing
two isomorphisms given in (\ref{CDCO})
and (\ref{COGHK}) yields the isomorphism
\begin{equation}\label{CDGHK}
\mathrm{C}(D)\stackrel{\sim}{\rightarrow}\mathrm{Gal}(H_\mathcal{O}/K),
\quad[Q]\mapsto(j(\mathcal{O})\mapsto
j([-\overline{\omega}_Q,\,1])).
\end{equation}
\par
Now, let $K$ be an imaginary quadratic field of discriminant
$d_K$ and $\mathcal{O}_K$ be its ring of integers.
If we set
\begin{equation}\label{tauK}
\tau_K=\left\{\begin{array}{ll}
\sqrt{d_K}/2 & \textrm{if}~d_K\equiv0\Mod{4},\\
(-1+\sqrt{d_K})/2 & \textrm{if}~d_K\equiv1\Mod{4},
\end{array}\right.
\end{equation}
then we get $\mathcal{O}_K=[\tau_K,\,1]$.
For a positive integer $N$ and $\mathfrak{n}=N\mathcal{O}_K$,
let $I_K(\mathfrak{n})$ be the group of fractional ideals
of $K$ relatively prime to $\mathfrak{n}$ and
$P_K(\mathfrak{n})$ be its subgroup of principal fractional ideals.
Furthermore, let
\begin{eqnarray*}
P_{K,\,\mathbb{Z}}(\mathfrak{n})&=&\{\nu\mathcal{O}_K~|~
\nu\in K^*~\textrm{such that}~\nu\equiv^*m\Mod{\mathfrak{n}}
~\textrm{for some integer $m$ prime to $N$}\},\\
P_{K,\,1}(\mathfrak{n})&=&\{\nu\mathcal{O}_K~|~
\nu\in K^*~\textrm{such that}~\nu\equiv^*1\Mod{\mathfrak{n}}\}
\end{eqnarray*}
which are subgroups of $P_K(\mathfrak{n})$.
As for the multiplicative congruence $\equiv^*$ modulo $\mathfrak{n}$,
we refer to \cite[$\S$IV.1]{Janusz}.
Then the ring class field $H_\mathcal{O}$ of order
$\mathcal{O}$ with conductor $N$ in $K$ and
the ray class field $K_\mathfrak{n}$ modulo $\mathfrak{n}$
are defined to be the unique abelian extensions of $K$ for which
the Artin map modulo $\mathfrak{n}$ induces the isomorphisms
\begin{equation*}
I_K(\mathfrak{n})/P_{K,\,\mathbb{Z}}(\mathfrak{n})
\simeq\mathrm{Gal}(H_\mathcal{O}/K)
\quad\textrm{and}\quad
I_K(\mathfrak{n})/P_{K,\,1}(\mathfrak{n})
\simeq\mathrm{Gal}(K_\mathfrak{n}/K),
\end{equation*}
respectively (\cite[$\S$8 and $\S$9]{Cox} and \cite[Chapter V]{Janusz}).
And, for a congruence subgroup $\Gamma$ of level $N$ in $\mathrm{SL}_2(\mathbb{Z})$, let
$\mathcal{F}_{\Gamma,\,\mathbb{Q}}$
be the field of meromorphic modular functions for $\Gamma$
whose Fourier expansions with respect to $q^{1/N}=e^{2\pi\mathrm{i}\tau/N}$ have rational coefficients and let
\begin{equation*}
K\mathcal{F}_{\Gamma,\,\mathbb{Q}}(\tau_K)
=K(h(\tau_K)~|~h\in\mathcal{F}_{\Gamma,\,\mathbb{Q}}~
\textrm{is finite at}~\tau_K).
\end{equation*}
Then it is a subfield of the maximal abelian extension $K^\mathrm{ab}$ of $K$ (\cite[Theorem 6.31 (i)]{Shimura}). In particular, for
the congruence subgroups
\begin{eqnarray*}
\Gamma_0(N)&=&\left\{\gamma\in\mathrm{SL}_2(\mathbb{Z})~|~
\gamma\equiv\begin{bmatrix}\mathrm{*}&\mathrm{*}\\
0&\mathrm{*}\end{bmatrix}\Mod{N M_2(\mathbb{Z})}\right\},\\
\Gamma_1(N)&=&\left\{\gamma\in\mathrm{SL}_2(\mathbb{Z})~|~
\gamma\equiv\begin{bmatrix}1&\mathrm{*}\\
0&1\end{bmatrix}\Mod{N M_2(\mathbb{Z})}\right\},
\end{eqnarray*}
we know that
\begin{equation}\label{specialization}
H_\mathcal{O}=K\mathcal{F}_{\Gamma_0(N),\,\mathbb{Q}}(\tau_K)
\quad\textrm{and}\quad
K_\mathfrak{n}=K\mathcal{F}_{\Gamma_1(N),\,\mathbb{Q}}(\tau_K)
\end{equation}
(\cite[Corollary 5.2]{C-K} and \cite[Theorem 3.4]{K-S13}).
On the other hand, one can naturally defines an equivalence
relation $\sim_\Gamma$ on the subset
\begin{equation}\label{QNdK}
\mathcal{Q}_N(d_K)=
\{ax^2+bxy+cy^2\in\mathcal{Q}(d_K)~|~\gcd(N,\,a)=1\}
\end{equation}
of $\mathcal{Q}(d_K)$ by
\begin{equation}\label{simG}
Q\sim_\Gamma Q'\quad\Longleftrightarrow\quad
Q'=Q^\gamma~\textrm{for some}~\gamma\in\Gamma.
\end{equation}
Observe that $\Gamma$ may not act on $\mathcal{Q}_N(d_K)$.
Here, by $Q^\gamma$ we mean the action of $\gamma$ as an element of $\mathrm{SL}_2(\mathbb{Z})$.
\par
For a subgroup $P$ of $I_K(\mathfrak{n})$ with $P_{K,\,1}(\mathfrak{n})\subseteq P\subseteq P_K(\mathfrak{n})$,
let $K_P$ be the abelian extension of $K$ so that
$I_K(\mathfrak{n})/P\simeq\mathrm{Gal}(K_P/K)$.
In this paper,
motivated by (\ref{CDGHK}) and (\ref{specialization})
we shall present several pairs of
$P$ and $\Gamma$ for which
\begin{itemize}
\item[(i)] $K_P=K\mathcal{F}_{\Gamma,\,\mathbb{Q}}(\tau_K)$,
\item[(ii)] $\mathcal{Q}_N(d_K)/\sim_\Gamma$ becomes a group isomorphic to $\mathrm{Gal}(K_P/K)$
via the isomorphism
\begin{equation}\label{desiredisomorphism}
\begin{array}{ccl}
\mathcal{Q}_N(d_K)/\sim_\Gamma&\stackrel{\sim}{\rightarrow}&
\mathrm{Gal}(K_P/K)\\
\left[Q\right]&\mapsto&(h(\tau_K)\mapsto h(-\overline{\omega}_Q)~|~
h\in\mathcal{F}_{\Gamma,\,\mathbb{Q}}~
\textrm{is finite at}~\tau_K)
\end{array}
\end{equation}
\end{itemize}
(Propositions \ref{KPKF}, \ref{satisfies} and
Theorems \ref{formclassgroup}, \ref{Galoisgroups}). This result would
be certain extension of Gauss' original work.
We shall also develop an algorithm of
finding distinct form classes in $\mathcal{Q}_N(d_K)/\sim_\Gamma$
and give a concrete example (Proposition \ref{algorithm} and Example \ref{example}). To this end, we shall apply
Shimura's theory
which links
the class field theory for imaginary quadratic fields and the theory of modular functions
(\cite[Chapter 6]{Shimura}).
And, we shall not only use but also improve the ideas of
our previous work \cite{E-K-S17}.
See Remarks \ref{difference}.

\section {Extended form class groups as ideal class groups}\label{sect2}

Let $K$ be an imaginary quadratic field of discriminant $d_K$
and $\tau_K$ be as in (\ref{tauK}).
And, let $N$ be a positive integer, $\mathfrak{n}=N\mathcal{O}_K$
and $P$ be a subgroup of $I_K(\mathfrak{n})$ satisfying
$P_{K,\,1}(\mathfrak{n})\subseteq P\subseteq P_K(\mathfrak{n})$.
Each subgroup $\Gamma$ of $\mathrm{SL}_2(\mathbb{Z})$ defines
an equivalence relation $\sim_\Gamma$ on
the set $\mathcal{Q}_N(d_K)$
described in (\ref{QNdK}) in the same manner as in (\ref{simG}).
In this section, we shall present
a necessary and sufficient condition for $\Gamma$ in such a way that
\begin{eqnarray*}
\phi_\Gamma:\mathcal{Q}_N(d_K)/\sim_\Gamma&\rightarrow&
I_K(\mathfrak{n})/P\\
\left[Q\right]~~~&\mapsto&[[\omega_Q,\,1]]
\end{eqnarray*}
becomes a well-defined bijection with $\omega_Q$ as in (\ref{wQ}). As mentioned in $\S$1, the lattice $[\omega_Q,\,1]=\mathbb{Z}\omega_Q+\mathbb{Z}$ is
a fractional ideal of $K$.
\par
The modular group $\mathrm{SL}_2(\mathbb{Z})$ acts on $\mathbb{H}$
from the left by fractional linear transformations. For each $Q\in\mathcal{Q}(d_K)$, let
$I_{\omega_Q}$ denote the isotropy subgroup of the point
$\omega_Q$ in $\mathrm{SL}_2(\mathbb{Z})$. In particular, if we let $Q_0$ be the principal form in $\mathcal{Q}(d_K)$
(\cite[p. 31]{Cox}), then we have $\omega_{Q_0}=\tau_K$ and
\begin{equation}\label{isotropy}
I_{\omega_{Q_0}}=\left\{
\begin{array}{ll}
\left\{\pm I_2\right\} & \textrm{if}~d_K\neq-4,\,-3,\vspace{0.1cm}\\
\left\{\pm I_2,\,\pm S\right\}
& \textrm{if}~d_K=-4,\vspace{0.1cm}\\
\left\{\pm I_2,\,\pm ST,\,
\pm(ST)^2\right\}
& \textrm{if}~d_K=-3
\end{array}
\right.
\end{equation}
where $S=\begin{bmatrix}
0&-1\\1&0\end{bmatrix}$ and $T=\begin{bmatrix}1&1\\
0&1\end{bmatrix}$.
Furthermore, we see that
\begin{equation}\label{otherisotropy}
I_{\omega_Q}=\{\pm I_2\}\quad
\textrm{if $\omega_Q$ is not equivalent to $\omega_{Q_0}$
under $\mathrm{SL}_2(\mathbb{Z})$}
\end{equation}
(\cite[Proposition 1.5 (c)]{Silverman}).
For any $\gamma=\begin{bmatrix}a&b\\c&d\end{bmatrix}\in
\mathrm{SL}_2(\mathbb{Z})$, let
\begin{equation*}
j(\gamma,\,\tau)=c\tau+d\quad(\tau\in\mathbb{H}).
\end{equation*}
One can readily check that if $Q'=Q^\gamma$, then
\begin{equation*}
\omega_Q=\gamma(\omega_{Q'})\quad\textrm{and}
\quad[\omega_Q,\,1]=\frac{1}{j(\gamma,\,\omega_{Q'})}
[\omega_{Q'},\,1].
\end{equation*}

\begin{lemma}\label{primein}
Let $Q=ax^2+bxy+cy^2\in\mathcal{Q}(d_K)$. Then
$\mathrm{N}_{K/\mathbb{Q}}([\omega_Q,\,1])=1/a$ and
\begin{equation*}
[\omega_Q,\,1]\in I_K(\mathfrak{n})\quad\Longleftrightarrow\quad
Q\in\mathcal{Q}_N(d_K).
\end{equation*}
\end{lemma}
\begin{proof}
See \cite[Lemma 2.3 (iii)]{E-K-S17}.
\end{proof}

\begin{lemma}\label{prime}
Let $Q=ax^2+bxy+cy^2\in\mathcal{Q}_N(d_K)$.
\begin{enumerate}
\item[\textup{(i)}]
For $u,\,v\in\mathbb{Z}$ not both zero, the fractional ideal $(u\omega_Q+v)\mathcal{O}_K$
is relatively prime to $\mathfrak{n}=N\mathcal{O}_K$ if and only if $\gcd(N,\,Q(v,\,-u))=1$.
\item[\textup{(ii)}] If $C\in P_K(\mathfrak{n})/P$, then
\begin{equation*}
C=[(u\omega_Q+v)\mathcal{O}_K]\quad
\textrm{for some}~u,\,v\in\mathbb{Z}~\textrm{not both zero such that}~
\gcd(N,\,Q(v,\,-u))=1.
\end{equation*}
\end{enumerate}
\end{lemma}
\begin{proof}
\begin{enumerate}
\item[(i)] See \cite[Lemma 4.1]{E-K-S17}
\item[(ii)] Since $P_K(\mathfrak{n})/P$ is a finite group, one can take an integral ideal
$\mathfrak{c}$ in the class $C$ (\cite[Lemma 2.3 in Chapter IV]{Janusz}).
Furthermore, since $\mathcal{O}_K=[a\omega_Q,\,1]$, we may express $\mathfrak{c}$ as
\begin{equation*}
\mathfrak{c}=(ka\omega_Q+v)\mathcal{O}_K\quad
\textrm{for some}~k,\,v\in\mathbb{Z}.
\end{equation*}
If we set $u=ka$, then we attain (ii) by (i).
\end{enumerate}
\end{proof}

\begin{proposition}\label{surjective}
If the map $\phi_\Gamma$ is well defined, then it is surjective.
\end{proposition}
\begin{proof}
Let
\begin{equation*}
\rho: I_K(\mathfrak{n})/P
\rightarrow I_K(\mathcal{O}_K)/P_K(\mathcal{O}_K)
\end{equation*}
be the natural homomorphism.
Since $I_K(\mathfrak{n})/P_K(\mathfrak{n})$ is isomorphic to  $I_K(\mathcal{O}_K)/P_K(\mathcal{O}_K)$
(\cite[Proposition 1.5 in Chapter IV]{Janusz}),
the homomorphism $\rho$ is surjective.
Here, we refer to the following commutative diagram:
\begin{figure}[H]
\begin{equation*}
\xymatrixcolsep{2pc}\xymatrix{
& I_K(\mathfrak{n})/P \ar@{->>}[ld] \ar[rd]^{~\rho} &\\
I_K(\mathfrak{n})/P_K(\mathfrak{n}) \ar[rr]^\sim &  & I_K(\mathcal{O}_K)/P_K(\mathcal{O}_K)
}
\end{equation*}
\caption{A commutative diagram of ideal class groups}\label{diagram}
\end{figure}
\noindent Let
\begin{equation*}
Q_1,\,Q_2,\,\ldots,\,Q_h\quad(\in\mathcal{Q}(d_K))
\end{equation*}
be reduced forms which represent all distinct classes in
$\mathrm{C}(d_K)=\mathcal{Q}(d_K)/\sim$ (\cite[Theorem 2.8]{Cox}). Take
$\gamma_1,\,\gamma_2,\,\ldots,\,\gamma_h\in\mathrm{SL}_2(\mathbb{Z})$
so that
\begin{equation*}
Q_i'=Q_i^{\gamma_i}\quad(i=1,\,2,\,\ldots,\,h)
\end{equation*}
belongs to $\mathcal{Q}_N(d_K)$ (\cite[Lemmas 2.3 and 2.25]{Cox}).
Then we get
\begin{equation*}
I_K(\mathcal{O}_K)/P_K(\mathcal{O}_K)=
\{[\omega_{Q_i'},\,1]P_K(\mathcal{O}_K)~|~i=1,\,2,\,\ldots,\,h\}
\quad\textrm{and}\quad[\omega_{Q_i'},\,1]\in I_K(\mathfrak{n})
\end{equation*}
by the isomorphism given in (\ref{CDCO}) (when
$D=d_K)$ and Lemma \ref{primein}.
Moreover, since $\rho$ is a surjection with $\mathrm{Ker}(\rho)=P_K(\mathfrak{n})/P$, we obtain the decomposition
\begin{equation}\label{decomp}
I_K(\mathfrak{n})/P=(P_K(\mathfrak{n})/P)\cdot
\{[[\omega_{Q_i'},\,1]]\in I_K(\mathfrak{n})/P~|~i=1,\,2,\,\ldots,\,h\}.
\end{equation}
\par
Now, let $C\in I_K(\mathfrak{n})/P$. By the decomposition (\ref{decomp}) and Lemma \ref{prime} (ii) we may express $C$ as
\begin{equation}\label{C}
C=\left[\frac{1}{u\omega_{Q_i'}+v}\,[\omega_{Q_i'},\,1]\right]
\end{equation}
for some $i\in\{1,\,2,\,\ldots,\,h\}$ and $u,\,v\in\mathbb{Z}$
not both zero with $\gcd(N,\,Q_i'(v,\,-u))=1$. Take any
$\sigma=\begin{bmatrix}\mathrm{*}&\mathrm{*}\\
\widetilde{u}&\widetilde{v}\end{bmatrix}\in\mathrm{SL}_2(\mathbb{Z})$ such that
$\sigma\equiv\begin{bmatrix}\mathrm{*}&\mathrm{*}\\
u&v\end{bmatrix}\Mod{N M_2(\mathbb{Z})}$.
We then derive that
\begin{eqnarray*}
C&=&\left[\frac{u\omega_{Q_i'}+v}{\widetilde{u}\omega_{Q_i'}+\widetilde{v}}
\,\mathcal{O}_K\right]C\quad
\textrm{because}~
\frac{u\omega_{Q_i'}+v}{\widetilde{u}\omega_{Q_i'}+\widetilde{v}}
\equiv^*1\Mod{\mathfrak{n}}~\textrm{and}~P_{K,\,1}(\mathfrak{n})
\subseteq P\\
&=&\left[\frac{1}{\widetilde{u}\omega_{Q_i'}+\widetilde{v}}
\,[\omega_{Q_i'},\,1]\right]\quad\textrm{by (\ref{C})}\\
&=&\left[\frac{1}{j(\sigma,\,\omega_{Q_i'})}\,[\omega_{Q_i'},\,1]
\right]\\
&=&[[\sigma(\omega_{Q_i'}),\,1]].
\end{eqnarray*}
Thus, if we put $Q=Q_i'^{\sigma^{-1}}$, then we obtain
\begin{equation*}
C=[[\omega_Q,\,1]]=\phi_\Gamma([Q]).
\end{equation*}
This prove that $\phi_\Gamma$ is surjective.
\end{proof}

\begin{proposition}\label{injective}
The map $\phi_\Gamma$ is a well-defined injection if and only if $\Gamma$ satisfies the following property:
\begin{equation}\label{P}
\begin{array}{c}
\textrm{Let $Q\in\mathcal{Q}_N(d_K)$ and
$\gamma\in\mathrm{SL}_2(\mathbb{Z})$ such that
$Q^{\gamma^{-1}}\in\mathcal{Q}_N(d_K)$. Then,}\\
j(\gamma,\,\omega_Q)\mathcal{O}_K\in P~\Longleftrightarrow~
\gamma\in\Gamma\cdot I_{\omega_Q}.
\end{array}
\end{equation}
\end{proposition}
\begin{proof}
Assume first that $\phi_\Gamma$ is a well-defined injection. Let
$Q\in\mathcal{Q}_N(d_K)$ and $\gamma\in\mathrm{SL}_2(\mathbb{Z})$
such that $Q^{\gamma^{-1}}\in\mathcal{Q}_N(d_K)$.
If we set $Q'=Q^{\gamma^{-1}}$, then we have
$Q=Q'^\gamma$ and so
\begin{equation}\label{wjw1}
[\omega_{Q'},\,1]=[\gamma(\omega_Q),\,1]=\frac{1}{j(\gamma,\,\omega_Q)}[\omega_Q,\,1].
\end{equation}
And, we deduce that
\begin{eqnarray*}
j(\gamma,\,\omega_Q)\mathcal{O}_K\in P&\Longleftrightarrow&
[[\omega_Q,\,1]]=
[[\omega_{Q'},\,1]]~\textrm{in}~I_K(\mathfrak{n})/P~
\textrm{by Lemma \ref{primein} and (\ref{wjw1})}\\
&\Longleftrightarrow&\phi_\Gamma([Q])=\phi_\Gamma([Q'])~
\textrm{by the definition of $\phi_\Gamma$}\\
&\Longleftrightarrow&[Q]=[Q']~\textrm{in
$\mathcal{Q}_N(d_K)/\sim_\Gamma$ since $\phi_\Gamma$ is injective}\\
&\Longleftrightarrow&Q'=Q^\alpha~\textrm{for some}~\alpha\in\Gamma\\
&\Longleftrightarrow&Q=Q^{\alpha\gamma}~\textrm{for some}~
\alpha\in\Gamma~\textrm{because}~Q'=Q^{\gamma^{-1}}\\
&\Longleftrightarrow&\alpha\gamma\in I_{\omega_Q}~
\textrm{for some}~\alpha\in\Gamma\\
&\Longleftrightarrow&\gamma\in\Gamma\cdot I_{\omega_Q}.
\end{eqnarray*}
Hence $\Gamma$ satisfies the property (\ref{P}).
\par
Conversely, assume that $\Gamma$ satisfies the property (\ref{P}).
To show that $\phi_\Gamma$ is well defined, suppose that
\begin{equation*}
[Q]=[Q']\quad\textrm{in}~\mathcal{Q}_N(d_K)/\sim_\Gamma
~\textrm{for some}~Q,\,Q'\in\mathcal{Q}_N(d_K).
\end{equation*}
Then we attain $Q=Q'^\alpha$ for some $\alpha\in\Gamma$ so that
\begin{equation}\label{wjw2}
[\omega_{Q'},\,1]=[\alpha(\omega_Q),\,1]=\frac{1}{j(\alpha,\,\omega_Q)}[\omega_Q,\,1].
\end{equation}
Now that $Q^{\alpha^{-1}}=Q'\in\mathcal{Q}_N(d_K)$
and $\alpha\in\Gamma\subseteq\Gamma\cdot
I_{\omega_Q}$, we achieve by the property (\ref{P}) that $j(\alpha,\,\omega_{Q})\mathcal{O}_K\in P$.
Thus we derive by
Lemma \ref{primein} and (\ref{wjw2}) that
\begin{equation*}
[[\omega_Q,\,1]]=[[\omega_{Q'},\,1]]
\quad\textrm{in}~I_K(\mathfrak{n})/P,
\end{equation*}
which claims that $\phi_\Gamma$ is well defined.
\par
On the other hand, in order to show that $\phi_\Gamma$ is injective, assume that
\begin{equation*}
\phi_\Gamma([Q])=\phi_\Gamma([Q'])\quad
\textrm{for some}~Q,\,Q'\in\mathcal{Q}_N(d_K).
\end{equation*}
Then we get
\begin{equation}\label{wlw}
[\omega_Q,\,1]=\lambda[\omega_{Q'},\,1]\quad
\textrm{for some}~\lambda\in K^*~\textrm{such that}~
\lambda\mathcal{O}_K\in P,
\end{equation}
from which it follows that
\begin{equation}\label{QQb}
Q=Q'^\gamma\quad\textrm{for some}~\gamma\in\mathrm{SL}_2(\mathbb{Z})
\end{equation}
by the isomorphism in (\ref{CDCO}) when $D=d_K$.
We then
derive by (\ref{wlw}) and (\ref{QQb}) that
\begin{equation*}
[\omega_{Q'},\,1]=[\gamma(\omega_Q),\,1]=
\frac{1}{j(\gamma,\,\omega_{Q})}
[\omega_Q,\,1]=\frac{\lambda}{j(\gamma,\,\omega_Q)}[\omega_{Q'},\,1]
\end{equation*}
and so $\lambda/j(\gamma,\,\omega_Q)\in\mathcal{O}_K^*$. Therefore we attain
\begin{equation*}
j(\gamma,\,\omega_Q)\mathcal{O}_K=\lambda\mathcal{O}_K\in P,
\end{equation*}
and hence $\gamma\in\Gamma\cdot I_{\omega_Q}$ by the
fact $Q^{\gamma^{-1}}=Q'\in\mathcal{Q}_N(d_K)$ and the property (\ref{P}). If we write
\begin{equation*}
\gamma=\alpha\beta\quad\textrm{for some}~\alpha\in\Gamma~
\textrm{and}~\beta\in I_{\omega_Q},
\end{equation*}
then we see by (\ref{QQb}) that
\begin{equation*}
Q=Q^{\beta^{-1}}=Q^{\gamma^{-1}\alpha}=Q'^\alpha.
\end{equation*}
This shows that
\begin{equation*}
[Q]=[Q']\quad\textrm{in}~\mathcal{Q}_N(d_K)/\sim_\Gamma,
\end{equation*}
which proves the injectivity of $\phi_\Gamma$.
\end{proof}

\begin{theorem}\label{formclassgroup}
The map $\phi_\Gamma$ is a well-defined bijection if and only
if $\Gamma$ satisfies the property \textup{(\ref{P})} stated in
\textup{Proposition \ref{injective}}. In this
case, we may regard the set $\mathcal{Q}_N(d_K)/\sim_\Gamma$
as a group isomorphic to the ideal class group
$I_K(\mathfrak{n})/P$.
\end{theorem}
\begin{proof}
We achieve the first assertion by Propositions \ref{surjective} and \ref{injective}.
Thus, in this case,
one can give a group structure on $\mathcal{Q}_N(d_K)/\sim_\Gamma$
through the bijection $\phi_\Gamma:
\mathcal{Q}_N(d_K)/\sim_\Gamma\rightarrow I_K(\mathfrak{n})/P$.
\end{proof}

\begin{remark}
By using the isomorphism given in (\ref{CDCO}) (when $D=d_K$) and Theorem \ref{formclassgroup}
we obtain the following commutative diagram:
\begin{figure}[h]
\begin{equation*}
\xymatrixcolsep{9pc}\xymatrix{
\mathcal{Q}_N(d_K)/\sim_\Gamma \ar[r]^\sim_{\phi_\Gamma}
\ar[d]_{\textrm{The natural map~}}&
I_K(\mathfrak{n})/P
\ar@{->>}[d]^{~\rho~\textrm{in Figure \ref{diagram}}}\\
\mathrm{C}(d_K)
\ar[r]^\sim_-{\textrm{The classical isomorphism in (\ref{CDCO})}} &
\mathrm{C}(\mathcal{O}_K)
}
\end{equation*}
\caption{The natural map between form class groups}\label{diagram2}
\end{figure}
\newpage
\noindent Therefore the natural map
$\mathcal{Q}_N(d_K)/\sim_\Gamma\rightarrow\mathrm{C}(d_K)$
is indeed a surjective homomorphism, which shows that
the group structure of $\mathcal{Q}_N(d_K)/\sim\Gamma$ is
not far from that of the classical form class group $\mathrm{C}(d_K)$.
\end{remark}

\section {Class field theory over imaginary quadratic fields}

In this section, we shall briefly review the class field theory over imaginary quadratic fields established by Shimura.
\par
For an imaginary quadratic field $K$, let $\mathbb{I}_K^\mathrm{fin}$ be the group of finite ideles of $K$ given by the restricted product
\begin{eqnarray*}
\mathbb{I}_K^\mathrm{fin}&=&{\prod_{\mathfrak{p}}}^\prime K_\mathfrak{p}^*\quad
\textrm{where $\mathfrak{p}$ runs
over all prime ideals of $\mathcal{O}_K$}\\
&=&\left\{s=(s_\mathfrak{p})\in\prod_\mathfrak{p}K_\mathfrak{p}^*~|~
s_\mathfrak{p}\in\mathcal{O}_{K_\mathfrak{p}}^*~\textrm{for
all but finitely many $\mathfrak{p}$}\right\}.
\end{eqnarray*}
As for the topology on $\mathbb{I}_K^\mathrm{fin}$
one can refer to \cite[p. 78]{Neukirch}.
Then, the classical class field theory of $K$ is explained by the exact sequence
\begin{equation*}
1\rightarrow K^*\rightarrow
\mathbb{I}_K^\mathrm{fin}\rightarrow\mathrm{Gal}(K^\mathrm{ab}/K)
\rightarrow 1
\end{equation*}
where $K^*$ maps into $\mathbb{I}_K^\mathrm{fin}$
through the diagonal embedding $\nu\mapsto(\nu,\,\nu,\,\nu,\,\ldots)$
(\cite[Chapter IV]{Neukirch}).
Setting
\begin{equation*}
\mathcal{O}_{K,\,p}=\mathcal{O}_K\otimes_\mathbb{Z}
\mathbb{Z}_p\quad\textrm{for each prime $p$}
\end{equation*}
we have
\begin{equation*}
\mathcal{O}_{K,\,p}
\simeq\prod_{\mathfrak{p}\,|\,p}
\mathcal{O}_{K_\mathfrak{p}}
\end{equation*}
(\cite[Proposition 4 in Chapter II]{Serre}).
Furthermore, if we let $\widehat{K}=K\otimes_\mathbb{Z}\widehat{\mathbb{Z}}$
with $\widehat{\mathbb{Z}}=\prod_p\mathbb{Z}_p$, then
\begin{eqnarray*}
\widehat{K}^*&=&{\prod_{p}}^\prime(K\otimes_\mathbb{Z}
\mathbb{Z}_p)^*\quad\textrm{where $p$ runs over all
rational primes}\\
&=&\left\{s=(s_p)\in\prod_p(K\otimes_\mathbb{Z}
\mathbb{Z}_p)^*~|~
s_p\in\mathcal{O}_{K,\,p}^*~
\textrm{for all but finitely many $p$}\right\}\\
&\simeq&\mathbb{I}_K^\mathrm{fin}
\end{eqnarray*}
(\cite[Exercise 15.12]{Cox} and \cite[Chapter II]{Serre}). Thus we may
use $\widehat{K}^*$ instead of $\mathbb{I}_K^\mathrm{fin}$ when we develop the class field theory of $K$.

\begin{proposition}\label{1-1}
There is a one-to-one correspondence via
the Artin map
between closed subgroups $J$ of $\widehat{K}^*$
of finite index containing $K^*$
and finite abelian extensions $L$ of $K$ such that
\begin{equation*}
\widehat{K}^*/J\simeq\mathrm{Gal}(L/K).
\end{equation*}
\end{proposition}
\begin{proof}
See \cite[Chapter IV]{Neukirch}.
\end{proof}

Let $N$ be a positive integer, $\mathfrak{n}=N\mathcal{O}_K$ and $s=(s_p)\in
\widehat{K}^*$.
For a prime $p$ and a prime ideal $\mathfrak{p}$ of $\mathcal{O}_K$ lying above $p$,
let $n_\mathfrak{p}(s)$ be a unique integer such that
$s_p\in
\mathfrak{p}^{n_\mathfrak{p}(s)}\mathcal{O}_{K_\mathfrak{p}}^*$.
We then regard $s\mathcal{O}_K$ as the fractional ideal
\begin{equation*}
s\mathcal{O}_K=
\prod_p\prod_{\mathfrak{p}\,|\,p}
\mathfrak{p}^{n_\mathfrak{p}(s)}\in I_K(\mathcal{O}_K).
\end{equation*}
By the approximation theorem
(\cite[Chapter IV]{Janusz}) one can take an element $\nu_s$ of $K^*$ such that
\begin{equation}\label{k}
\nu_ss_p\in 1+N\mathcal{O}_{K,\,p}\quad\textrm{for
all}~p\,|\,N.
\end{equation}

\begin{proposition}\label{rayidele}
We get a well-defined surjective homomorphism
\begin{eqnarray*}
\phi_\mathfrak{n}~:~\widehat{K}^*&\rightarrow& I_K(\mathfrak{n})/P_{K,\,1}(\mathfrak{n})\\
s~~&\mapsto&~~~[\nu_ss\mathcal{O}_K]
\end{eqnarray*}
with kernel
\begin{equation*}
J_\mathfrak{n}=
K^*\left(
\prod_{p\,|\,N}(1+N\mathcal{O}_{K,\,p})\times
\prod_{p\,\nmid\,N}\mathcal{O}_{K,\,p}^*\right).
\end{equation*}
Thus
$J_\mathfrak{n}$ corresponds to the ray class field $K_\mathfrak{n}$.
\end{proposition}
\begin{proof}
See \cite[Exercises 15.17 and 15.18]{Cox}.
\end{proof}

Let $\mathcal{F}_N$ be the field of meromorphic modular functions of level $N$
whose Fourier expansions with respect to
$q^{1/N}$
have coefficients in the $N$th cyclotomic field $\mathbb{Q}(\zeta_N)$
with $\zeta_N=e^{2\pi\mathrm{i}/N}$.
Then $\mathcal{F}_N$ is a Galois extension of $\mathcal{F}_1$
with $\mathrm{Gal}(\mathcal{F}_N/\mathcal{F}_1)\simeq
\mathrm{GL}_2(\mathbb{Z}/N\mathbb{Z})/\{\pm I_2\}$
(\cite[Chapter 6]{Shimura}).

\begin{proposition}\label{Galoisdecomposition}
There is a decomposition
\begin{equation*}
\mathrm{GL}_2(\mathbb{Z}/N\mathbb{Z})/\{\pm I_2\}
=\left\{\pm\begin{bmatrix}1&0\\0&d\end{bmatrix}~|~
d\in(\mathbb{Z}/N\mathbb{Z})^*\right\}/\{\pm I_2\}
\cdot\mathrm{SL}_2(\mathbb{Z}/N\mathbb{Z})/\{\pm I_2\}.
\end{equation*}
Let $h(\tau)$ be an element of $\mathcal{F}_N$ whose Fourier expansion is given by
\begin{equation*}
h(\tau)=\sum_{n\gg-\infty}c_nq^{n/N}\quad(c_n\in\mathbb{Q}(\zeta_N)).
\end{equation*}
\begin{enumerate}
\item[\textup{(i)}]
If $\alpha=\begin{bmatrix}1&0\\0&d\end{bmatrix}$ with
$d\in(\mathbb{Z}/N\mathbb{Z})^*$, then
\begin{equation*}
h(\tau)^\alpha=
\sum_{n\gg-\infty}c_n^{\sigma_d}q^{n/N}
\end{equation*}
where $\sigma_d$ is the automorphism of $\mathbb{Q}(\zeta_N)$
defined by $\zeta_N^{\sigma_d}=\zeta_N^d$.
\item[\textup{(ii)}]
If $\beta\in\mathrm{SL}_2(\mathbb{Z}/N\mathbb{Z})/\{\pm I_2\}$, then
\begin{equation*}
h(\tau)^\beta=h(\gamma(\tau))
\end{equation*}
where $\gamma$ is any
element of $\mathrm{SL}_2(\mathbb{Z})$ which
maps to $\beta$ through the reduction
$\mathrm{SL}_2(\mathbb{Z})\rightarrow
\mathrm{SL}_2(\mathbb{Z}/N\mathbb{Z})/\{\pm I_2\}$.
\end{enumerate}
\end{proposition}
\begin{proof}
See \cite[Proposition 6.21]{Shimura}.
\end{proof}

If we let $\widehat{\mathbb{Q}}=\mathbb{Q}\otimes_\mathbb{Z}\widehat{\mathbb{Z}}$
and $\displaystyle\mathcal{F}=\bigcup_{N=1}^\infty\mathcal{F}_N$,
then we attain the exact sequence
\begin{equation*}
1\rightarrow\mathbb{Q}^*\rightarrow\mathrm{GL}_2(\widehat{\mathbb{Q}})
\rightarrow\mathrm{Gal}(\mathcal{F}/\mathbb{Q})\rightarrow1
\end{equation*}
(\cite[Chaper 7]{Lang87} or \cite[Chapter 6]{Shimura}).
Here, we note that
\begin{eqnarray*}
\mathrm{GL}_2(\widehat{\mathbb{Q}})&=&
{\prod_p}^\prime\mathrm{GL}_2(\mathbb{Q}_p)
\quad\textrm{where $p$ runs over all
rational primes}\\
&=&
\{\gamma=(\gamma_p)\in\prod_p\mathrm{GL}_2(\mathbb{Q}_p)~|~
\gamma_p\in\mathrm{GL}_2(\mathbb{Z}_p)~
\textrm{for all but finitely many $p$}\}
\end{eqnarray*}
(\cite[Exercise 15.4]{Cox}) and $\mathbb{Q}^*$ maps into $\mathrm{GL}_2(\widehat{\mathbb{Q}})$ through the diagonal embedding.
For $\omega\in K\cap\mathbb{H}$, we define a normalized  embedding
\begin{equation*}
q_\omega:K^*\rightarrow\mathrm{GL}_2^+(\mathbb{Q})
\end{equation*}
by the relation
\begin{equation}\label{defq}
\nu\begin{bmatrix}\tau_K\\1\end{bmatrix}=
q_\omega(\nu)\begin{bmatrix}\tau_K\\1\end{bmatrix}\quad(\nu\in K^*).
\end{equation}
By continuity, $q_\omega$ can be extended to an embedding
\begin{equation*}
q_{\omega,\,p}:(K\otimes_\mathbb{Z}\mathbb{Z}_p)^*\rightarrow\mathrm{GL}_2(\mathbb{Q}_p)
\quad\textrm{for each prime $p$}
\end{equation*}
and hence to an embedding
\begin{equation*}
q_\omega:\widehat{K}^*\rightarrow\mathrm{GL}_2(\widehat{\mathbb{Q}}).
\end{equation*}
Let $\min(\tau_K,\,\mathbb{Q})=x^2+b_Kx+c_K$ ($\in\mathbb{Z}[x]$).
Since $K\otimes_\mathbb{Z}\mathbb{Z}_p=\mathbb{Q}_p\tau_K+\mathbb{Q}_p$
for each prime $p$,
one can deduce that if
$s=(s_p)\in\widehat{K}^*$ with
$s_p=u_p\tau_K+v_p$ ($u_p,\,v_p\in\mathbb{Q}_p$),
then
\begin{equation}\label{gamma_p}
q_{\tau_K}(s)=(\gamma_p)\quad\textrm{with}~\gamma_p=
\begin{bmatrix}v_p-b_Ku_p & -c_Ku_p\\u_p&v_p\end{bmatrix}.
\end{equation}
\par
By utilizing the concept of canonical models of modular curves, Shimura achieved the following remarkable results.

\begin{proposition}[Shimura's reciprocity law]\label{reciprocity}
Let $s\in\widehat{K}^*$, $\omega\in K\cap\mathbb{H}$ and
$h\in\mathcal{F}$ be finite at $\omega$. Then $h(\omega)$ lies
in $K^\mathrm{ab}$ and satisfies
\begin{equation*}
h(\omega)^{[s^{-1},\,K]}=h(\tau)^{q_\omega(s)}|_{\tau=\omega}
\end{equation*}
where $[\,\cdot,\,K]$ is the Artin map for $K$.
\end{proposition}
\begin{proof}
See \cite[Theorem 6.31 (i)]{Shimura}.
\end{proof}

\begin{proposition}\label{model}
Let $S$ be an open subgroup of $\mathrm{GL}_2
(\widehat{\mathbb{Q}})$ containing $\mathbb{Q}^*$ such that
$S/\mathbb{Q}^*$ is compact.
Let
\begin{eqnarray*}
\Gamma_S&=&S\cap\mathrm{GL}_2^+(\mathbb{Q}),\\
\mathcal{F}_S&=&\{h\in\mathcal{F}~|~h^\gamma=h~\textrm{for all}~\gamma\in S\},\\
k_S&=&\{\nu\in\mathbb{Q}^\mathrm{ab}~|~
\nu^{[s,\,\mathbb{Q}]}=\nu~\textrm{for all}~
s\in\mathbb{Q}^*\det(S)
\subseteq\widehat{\mathbb{Q}}^*\}
\end{eqnarray*}
where $[\,\cdot,\,\mathbb{Q}]$ is the Artin map for $\mathbb{Q}$.
Then,
\begin{enumerate}
\item[\textup{(i)}] $\Gamma_S/\mathbb{Q}^*$
is a Fuchsian group of the first kind commensurable with
$\mathrm{SL}_2(\mathbb{Z})/\{\pm I_2\}$.
\item[\textup{(ii)}] $\mathbb{C}\mathcal{F}_S$
is the field of meromorphic modular functions for
$\Gamma_S/\mathbb{Q}^*$.
\item[\textup{(iii)}]
$k_S$ is algebraically closed in $\mathcal{F}_S$.
\item[\textup{(iv)}] If $\omega\in K\cap\mathbb{H}$, then
the subgroup $K^*q_\omega^{-1}(S)$ of
$\widehat{K}^*$ corresponds to
the subfield
\begin{equation*}
 K\mathcal{F}_S(\omega)=K(h(\omega)~|~h\in\mathcal{F}_S~
\textrm{is finite at}~\omega)
\end{equation*}
of $K^\mathrm{ab}$ in view of \textup{Proposition \ref{1-1}}.
\end{enumerate}
\end{proposition}
\begin{proof}
See \cite[Propositions 6.27 and 6.33]{Shimura}.
\end{proof}

\begin{remark}\label{overQ}
In particular, if $k_S=\mathbb{Q}$, then
$\mathcal{F}_S=\mathcal{F}_{\Gamma_S,\,\mathbb{Q}}$
(\cite[Exercise 6.26]{Shimura}).
\end{remark}

\section {Construction of class invariants}\label{classinvariants}

Let $K$ be an imaginary quadratic field, $N$ be a positive integer
and $\mathfrak{n}=N\mathcal{O}_K$.
From now on, let $T$ be a subgroup of $(\mathbb{Z}/N\mathbb{Z})^*$ and $P$ be
a subgroup of $P_K(\mathfrak{n})$ containing
$P_{K,\,1}(\mathfrak{n})$ given by
\begin{eqnarray*}
P&=&\langle\nu\mathcal{O}_K~|~
\nu\in\mathcal{O}_K-\{0\}~\textrm{such that}~\nu\equiv t\Mod{\mathfrak{n}}~
\textrm{for some}~t\in T\rangle\\
&=&\{\lambda\mathcal{O}_K~|~\lambda\in K^*~
\textrm{such that}~\lambda\equiv^*t\Mod{\mathfrak{n}}~
\textrm{for some}~t\in T\}.\nonumber
\end{eqnarray*}
Let $\mathrm{Cl}(P)$ denote the ideal class group
\begin{equation*}
\mathrm{Cl}(P)=I_K(\mathfrak{n})/P
\end{equation*}
and $K_P$ be its corresponding class field of $K$ with
$\mathrm{Cl}(P)\simeq\mathrm{Gal}(K_P/K)$.
Furthermore, let
\begin{equation*}
\Gamma=\left\{\gamma\in\mathrm{SL}_2(\mathbb{Z})~|~
\gamma\equiv\begin{bmatrix}t^{-1}&\mathrm{*}\\
0&t\end{bmatrix}\Mod{N M_2(\mathbb{Z})}~
\textrm{for some}~t\in T\right\}
\end{equation*}
where $t^{-1}$ stands for an integer such that $tt^{-1}\equiv1\Mod{N}$.
In this section, for a given $h\in\mathcal{F}_{\Gamma,\,\mathbb{Q}}$ we shall define a class
invariant $h(C)$ for each class $C\in I_K(\mathfrak{n})/P$.

\begin{lemma}\label{subgroup}
The field $K_P$ corresponds to the subgroup
\begin{equation*}
\bigcup_{t\in T}K^*\left(\prod_{p\,|\,N}(t+N\mathcal{O}_{K,\,p})
\times\prod_{p\,\nmid\,N}\mathcal{O}_{K,\,p}^*\right)
\end{equation*}
of $\widehat{K}^*$
in view of \textup{Proposition \ref{1-1}}.
\end{lemma}
\begin{proof}
We adopt the notations in Proposition \ref{rayidele}.
Given $t\in T$, let $t^{-1}$ be an integer such that
$tt^{-1}\equiv1\Mod{N}$.
Let $s=s(t)=(s_p)\in
\widehat{K}^*$ be given by
\begin{equation*}
s_p=\left\{\begin{array}{ll}
t^{-1} & \textrm{if}~p\,|\,N,\\
1 & \textrm{if}~p\nmid N.
\end{array}\right.
\end{equation*}
Then one can take $\nu_s=t$ so as to have (\ref{k}), and hence
\begin{equation}\label{st}
\phi_\mathfrak{n}(s)=[ts\mathcal{O}_K]=[t\mathcal{O}_K].
\end{equation}
Since $P$ contains $P_{K,\,1}(\mathfrak{n})$,
we obtain
$K_P\subseteq K_\mathfrak{n}$ and $\mathrm{Gal}(K_\mathfrak{n}/K_P)\simeq P/P_{K,\,1}(\mathfrak{n})$.
Thus we achieve by Proposition \ref{rayidele} that the field $K_P$ corresponds to
\begin{eqnarray*}
\phi_\mathfrak{n}^{-1}(P/P_{K,\,1}(\mathfrak{n}))&=&
\phi_\mathfrak{n}^{-1}\left(\bigcup_{t\in T}[t\mathcal{O}_K]\right)\quad\textrm{by
the definitions of $P_{K,\,1}(\mathfrak{n})$ and $P$}\\
&=&\bigcup_{t\in T}s(t)J_\mathfrak{n}
\quad\textrm{by (\ref{st}) and the fact $J_\mathfrak{n}=\mathrm{Ker}(\phi_\mathfrak{n})$}\\
&=&\bigcup_{t\in T}K^*\left(\prod_{p\,|\,N}(t^{-1}+N\mathcal{O}_{K,\,p})
\times\prod_{p\,\nmid\,N}\mathcal{O}_{K,\,p}^*\right)\\
&=&\bigcup_{t\in T}K^*\left(\prod_{p\,|\,N}(t+N\mathcal{O}_{K,\,p})
\times\prod_{p\,\nmid\,N}\mathcal{O}_{K,\,p}^*\right).
\end{eqnarray*}
\end{proof}

\begin{proposition}\label{KPKF}
We have $K_P=K\mathcal{F}_{\Gamma,\,\mathbb{Q}}(\tau_K)$.
\end{proposition}
\begin{proof}
Let $S=\mathbb{Q}^*W$ ($\subseteq\mathrm{GL}_2(\widehat{\mathbb{Q}})$) with
\begin{equation*}
W=\bigcup_{t\in T}\left\{
\gamma=(\gamma_p)\in\prod_p\mathrm{GL}_2(\mathbb{Z}_p)~|~
\gamma_p\equiv\begin{bmatrix}\mathrm{*}&\mathrm{*}\\
0&t\end{bmatrix}\Mod{N M_2(\mathbb{Z}_p)}~\textrm{for all $p$}\right\}.
\end{equation*}
Following the notations in Proposition \ref{model}
one can readily show that
\begin{equation*}
\Gamma_S=
\mathbb{Q}^*\left\{
\gamma\in\mathrm{SL}_2(\mathbb{Z})~|~\gamma
\equiv\begin{bmatrix}\mathrm{*}&\mathrm{*}\\0&t\end{bmatrix}
\Mod{N M_2(\mathbb{Z})}~\textrm{for some}~t\in T\right\}\quad
\textrm{and}\quad
\det(W)=\widehat{\mathbb{Z}}^*.
\end{equation*}
It then follows that
$\Gamma_S/\mathbb{Q}^*\simeq\Gamma/\{\pm I_2\}$
and $k_S=\mathbb{Q}$,
and hence
\begin{equation}\label{F_S}
\mathcal{F}_S=\mathcal{F}_{\Gamma,\,\mathbb{Q}}
\end{equation}
by Proposition \ref{model} (ii) and Remark \ref{overQ}.
Furthermore, we deduce that
\begin{eqnarray*}
K^*q_{\tau_K}^{-1}(S)&=&K^*\{s=(s_p)\in\widehat{K}^*~|~q_{\tau_K}(s)\in \mathbb{Q}^*W\}\\
&=&K^*\{s=(s_p)\in\widehat{K}^*~|~q_{\tau_K}(s)\in W\}
\quad\textrm{since}~q_{\tau_K}(r)=rI_2~\textrm{for every}~r\in\mathbb{Q}^*~\textrm{by (\ref{defq})}\\
&=&K^*\{s=(s_p)\in\widehat{K}^*~|~
s_p=u_p\tau_K+v_p~\textrm{with}~u_p,\,v_p\in\mathbb{Q}_p
~\textrm{such that}\\
&&\hspace{3.4cm}\gamma_p=\left[\begin{smallmatrix}v_p-b_Ku_p & -c_Ku_p\\u_p&v_p\end{smallmatrix}\right]
\in W~\textrm{for all $p$}\}\quad\textrm{by (\ref{gamma_p})}\\
&=&\bigcup_{t\in T}K^*\{s=(s_p)\in\widehat{K}^*~|~
s_p=u_p\tau_K+v_p~\textrm{with}~u_p,\,v_p\in\mathbb{Z}_p
~\textrm{such that}\\
&&\hspace{4cm}\gamma_p\in\mathrm{GL}_2(\mathbb{Z}_p)
~\textrm{and}~\gamma_p\equiv
\left[\begin{smallmatrix}\mathrm{*}&\mathrm{*}\\0&t\end{smallmatrix}\right]
\Mod{N M_2(\mathbb{Z}_p)}~\textrm{for all $p$}\}\\
&=&\bigcup_{t\in T}K^*\{s=(s_p)\in\widehat{K}^*~|~
s_p=u_p\tau_K+v_p~\textrm{with}~u_p,\,v_p\in\mathbb{Z}_p
~\textrm{such that}\\
&&\hspace{4cm}\det(\gamma_p)=(u_p\tau_K+v_p)(u_p\overline{\tau}_K+v_p)\in\mathbb{Z}_p^*,\\
&&\hspace{4cm}u_p\equiv0\Mod{N\mathbb{Z}_p}~
\textrm{and}~v_p\equiv t\Mod{N\mathbb{Z}_p}~\textrm{for all $p$}\}\\
&=&\bigcup_{t\in T}K^*\left(
\prod_{p\,|\,N}(t+N\mathcal{O}_{K,\,p})\times
\prod_{p\,\nmid\,N}\mathcal{O}_{K,\,p}^*\right).
\end{eqnarray*}
Therefore we conclude by Proposition \ref{model} (iv), (\ref{F_S}) and Lemma \ref{subgroup} that
\begin{equation*}
K_P=K\mathcal{F}_{\Gamma,\,\mathbb{Q}}(\tau_K).
\end{equation*}
\end{proof}

Let $C\in\mathrm{Cl}(P)$.
Take an integral ideal $\mathfrak{a}$ in
the class $C$, and let
$\xi_1$ and $\xi_2$ be elements of $K^*$ so that
\begin{equation*}
\mathfrak{a}^{-1}=[\xi_1,\,\xi_2]
\quad
\textrm{and}\quad
\xi=\frac{\xi_1}{\xi_2}\in\mathbb{H}.
\end{equation*}
Since $\mathcal{O}_K=[\tau_K,\,1]\subseteq\mathfrak{a}^{-1}$ and
$\xi\in\mathbb{H}$, one can express
\begin{equation}\label{A}
\begin{bmatrix}\tau_K\\1\end{bmatrix}=A\begin{bmatrix}
\xi_1\\\xi_2\end{bmatrix}\quad\textrm{for some}~
A\in M_2^+(\mathbb{Z}).
\end{equation}
We then attain by taking determinant and squaring
\begin{equation*}
\begin{bmatrix}
\tau_K&\overline{\tau}_K\\1&1\end{bmatrix}
=A\begin{bmatrix}\xi_1&\overline{\xi}_1\\
\xi_2&\overline{\xi}_2\end{bmatrix}
\end{equation*}
that
\begin{equation*}
d_K=\det(A)^2\mathrm{N}_{K/\mathbb{Q}}(\mathfrak{a})^{-2}d_K
\end{equation*}
(\cite[Chapter III]{Lang94}).
Hence, $\det(A)=\mathrm{N}_{K/\mathbb{Q}}(\mathfrak{a})$
which is relatively prime to $N$.
For $\alpha\in M_2(\mathbb{Z})$ with $\gcd(N,\,\det(\alpha))=1$, we shall denote by
$\widetilde{\alpha}$ its reduction
onto $\mathrm{GL}_2(\mathbb{Z}/N\mathbb{Z})/\{\pm I_2\}$
($\simeq\mathrm{Gal}(\mathcal{F}_N/\mathcal{F}_1)$).

\begin{definition}\label{invariant}
Let $h\in\mathcal{F}_{\Gamma,\,\mathbb{Q}}$ \textup{($\subseteq
\mathcal{F}_N$)}.
With the notations as above, we define
\begin{equation*}
h(C)=h(\tau)^{\widetilde{A}}|_{\tau=\xi}
\end{equation*}
if it is finite.
\end{definition}

\begin{proposition}
If $h(C)$ is finite, then it depends only on the class $C$
regardless of the
choice of $\mathfrak{a}$, $\xi_1$ and $\xi_2$.
\end{proposition}
\begin{proof}
Let $\mathfrak{a}'$ be also an integral ideal in $C$.
Take any $\xi_1',\,\xi_2'\in K^*$ so that
\begin{equation}\label{a'}
\mathfrak{a}'^{-1}=[\xi_1',\,\xi_2']\quad
\textrm{and}\quad \xi'=\frac{\xi_1'}{\xi_2'}\in\mathbb{H}.
\end{equation}
Since $\mathcal{O}_K\subseteq\mathfrak{a}'^{-1}$ and
$\xi'\in\mathbb{H}$, we may write
\begin{equation}\label{A'}
\begin{bmatrix}\tau_K\\1\end{bmatrix}=A'\begin{bmatrix}
\xi_1'\\\xi_2'\end{bmatrix}
\quad\textrm{for some}~
A'\in M_2^+(\mathbb{Z}).
\end{equation}
Now that $[\mathfrak{a}]=[\mathfrak{a}']=C$,
we have
\begin{equation*}
\mathfrak{a}'=\lambda\mathfrak{a}\quad
\textrm{with}~\lambda\in K^*~\textrm{such that}~\lambda\equiv^*t\Mod{\mathfrak{n}}~
\textrm{for some}~t\in T.
\end{equation*}
Then it follows that
\begin{equation}\label{a-l-a-}
\mathfrak{a}'^{-1}=\lambda^{-1}\mathfrak{a}^{-1}=
[\lambda^{-1}\xi_1,\,\lambda^{-1}\xi_2]
\quad\textrm{and}\quad\frac{\lambda^{-1}\xi_1}
{\lambda^{-1}\xi_2}=\xi.
\end{equation}
And, we obtain by (\ref{a'}) and (\ref{a-l-a-}) that
\begin{equation}\label{B}
\begin{bmatrix}\xi_1'\\\xi_2'\end{bmatrix}
=B\begin{bmatrix}\lambda^{-1}\xi_1\\
\lambda^{-1}\xi_2\end{bmatrix}\quad
\textrm{for some}~B\in\mathrm{SL}_2(\mathbb{Z})
\end{equation}
and
\begin{equation}\label{x'Bx}
\xi'=B(\xi).
\end{equation}
On the other hand,
consider $t$ as an integer whose reduction modulo $N$ belongs to $T$.
Since
$\mathfrak{a},\,\mathfrak{a}'=\lambda\mathfrak{a}\subseteq\mathcal{O}_K$,
we see that
$(\lambda-t)\mathfrak{a}$ is an integral ideal.
Moreover, since $\lambda\equiv^*t\Mod{\mathfrak{n}}$
and $\mathfrak{a}$ is relatively prime to $\mathfrak{n}$, we get
$(\lambda-t)\mathfrak{a}\subseteq \mathfrak{n}=N\mathcal{O}_K$,
and hence
\begin{equation*}
(\lambda-t)\mathcal{O}_K\subseteq N\mathfrak{a}^{-1}.
\end{equation*}
Thus we attain by the facts $\mathcal{O}_K=[\tau_K,\,1]$
and $\mathfrak{a}^{-1}=[\xi_1,\,\xi_2]$ that
\begin{equation}\label{A''}
\begin{bmatrix}
(\lambda-t)\tau_K\\\lambda-t
\end{bmatrix}
=A''
\begin{bmatrix}N\xi_1\\N\xi_2\end{bmatrix}
\quad\textrm{for some}~A''\in M_2^+(\mathbb{Z}).
\end{equation}
We then derive that
\begin{eqnarray*}
NA''\begin{bmatrix}\xi_1\\\xi_2\end{bmatrix}&=&
\lambda\begin{bmatrix}\tau_K\\1\end{bmatrix}
-t\begin{bmatrix}\tau_K\\1\end{bmatrix}\quad\textrm{by (\ref{A''})}\\
&=&\lambda A'\begin{bmatrix}\xi_1'\\\xi_2'\end{bmatrix}-
tA\begin{bmatrix}\xi_1\\\xi_2\end{bmatrix}
\quad\textrm{by (\ref{A}) and (\ref{A'})}\\
&=&A'B\begin{bmatrix}\xi_1\\\xi_2\end{bmatrix}-
tA\begin{bmatrix}\xi_1\\\xi_2\end{bmatrix}
\quad\textrm{by (\ref{B})}.
\end{eqnarray*}
This yields $A'B-tA\equiv O\Mod{N M_2(\mathbb{Z})}$ and so
\begin{equation}\label{AtAB}
A'\equiv tAB^{-1}\Mod{N M_2(\mathbb{Z})}.
\end{equation}
Therefore we establish by Proposition \ref{Galoisdecomposition} that
\begin{eqnarray*}
h(\tau)^{\widetilde{A'}}|_{\tau=\xi'}&=&
h(\tau)^{\widetilde{tAB^{-1}}}|_{\tau=\xi'}\quad
\textrm{by (\ref{AtAB})}\\
&&\hspace{2.8cm}\textrm{where $\widetilde{~\cdot~}$ means the reduction onto $\mathrm{GL}_2(\mathbb{Z}/N\mathbb{Z})/\{\pm I_2\}$}
\\
&=&h(\tau)^{\widetilde{
\left[\begin{smallmatrix}1&0\\0&t^2\end{smallmatrix}\right]}
\widetilde{\left[\begin{smallmatrix}t&0\\0&t^{-1}\end{smallmatrix}\right]}
\widetilde{AB^{-1}}}|_{\tau=\xi'}\\
&=&h(\tau)^{\widetilde{
\left[\begin{smallmatrix}t&0\\0&t^{-1}\end{smallmatrix}\right]}
\widetilde{AB^{-1}}}|_{\tau=\xi'}\quad
\textrm{because $h(\tau)$ has rational Fourier coefficients}\\
&=&h(\tau)^{\widetilde{AB^{-1}}}|_{\tau=\xi'}\quad
\textrm{since $h(\tau)$ is modular for $\Gamma$}\\
&=&h(\tau)^{\widetilde{A}}|_{\tau=B^{-1}(\xi')}\\
&=&h(\tau)^{\widetilde{A}}|_{\tau=\xi}\quad
\textrm{by (\ref{x'Bx})}.
\end{eqnarray*}
This prove that $h(C)$ depends only on the class $C$.
\end{proof}

\begin{remark}\label{identityinvariant}
If we let $C_0$ be the identity class in $\mathrm{Cl}(P)$, then
we have $h(C_0)=h(\tau_K)$.
\end{remark}

\begin{proposition}\label{transformation}
Let $C\in\mathrm{Cl}(P)$ and $h\in\mathcal{F}_{\Gamma,\,\mathbb{Q}}$. If $h(C)$ is finite, then
it belongs to $K_P$ and satisfies
\begin{equation*}
h(C)^{\sigma(C')}=h(CC')\quad\textrm{for all}~C'\in
\mathrm{Cl}(P)
\end{equation*}
where $\sigma:\mathrm{Cl}(P)\rightarrow\mathrm{Gal}(K_P/K)$
is the isomorphism induced from the Artin map.
\end{proposition}
\begin{proof}
Let $\mathfrak{a}$ be an integral ideal in $C$ and $\xi_1,\,\xi_2\in K^*$ such that
\begin{equation}\label{ax}
\mathfrak{a}^{-1}=[\xi_1,\,\xi_2]\quad
\textrm{with}~\xi=\frac{\xi_1}{\xi_2}\in\mathbb{H}.
\end{equation}
Then we have
\begin{equation}\label{tAx}
\begin{bmatrix}\tau_K\\1\end{bmatrix}=A\begin{bmatrix}
\xi_1\\\xi_2\end{bmatrix}\quad\textrm{for some}~A\in M_2^+(\mathbb{Z}).
\end{equation}
Furthermore, let $\mathfrak{a}'$ be an integral ideal in $C'$ and $\xi_1'',\,\xi_2''\in K^*$ such that
\begin{equation}\label{aax}
(\mathfrak{a}\mathfrak{a}')^{-1}=[\xi_1'',\,\xi_2'']
\quad\textrm{with}~\xi''=\frac{\xi_1''}{\xi_2''}\in\mathbb{H}.
\end{equation}
Since $\mathfrak{a}^{-1}\subseteq
(\mathfrak{a}\mathfrak{a}')^{-1}$ and $\xi''\in\mathbb{H}$, we get
\begin{equation}\label{xBx}
\begin{bmatrix}\xi_1\\\xi_2\end{bmatrix}=B\begin{bmatrix}
\xi_1''\\\xi_2''\end{bmatrix}\quad
\textrm{for some}~B\in M_2^+(\mathbb{Z}),
\end{equation}
and so it follows from (\ref{tAx}) that
\begin{equation}\label{tABx}
\begin{bmatrix}\tau_K\\1\end{bmatrix}=
AB\begin{bmatrix}\xi_1''\\\xi_2''\end{bmatrix}.
\end{equation}
Let $s=(s_p)$ be an idele in $\widehat{K}^*$ satisfying
\begin{equation}\label{s1pN}
\left\{\begin{array}{ll}
s_p=1 & \textrm{if}~p\,|\,N,\\
s_p\mathcal{O}_{K,\,p}=\mathfrak{a}_p' & \textrm{if}~
p\nmid N
\end{array}\right.
\end{equation}
where $\mathfrak{a}_p'=\mathfrak{a}'\otimes_\mathbb{Z}
\mathbb{Z}_p$. Since
$\mathfrak{a}'$ is relatively prime to $\mathfrak{n}=N\mathcal{O}_K$, we
obtain by (\ref{s1pN}) that
\begin{equation}\label{sOa}
s_p^{-1}\mathcal{O}_{K,\,p}=\mathfrak{a}_p'^{-1}\quad
\textrm{for all $p$}.
\end{equation}
Now, we see that
\begin{equation*}
q_{\xi,\,p}(s_p^{-1})
\begin{bmatrix}\xi_1\\\xi_2\end{bmatrix}
=\xi_2q_{\xi,\,p}(s_p^{-1})
\begin{bmatrix}\xi\\1\end{bmatrix}
=\xi_2s_p^{-1}\begin{bmatrix}\xi\\1\end{bmatrix}=
s_p^{-1}\begin{bmatrix}\xi_1\\\xi_2\end{bmatrix},
\end{equation*}
which shows by (\ref{ax}) and (\ref{sOa}) that $q_{\xi,\,p}(s_p^{-1})\begin{bmatrix}\xi_1\\\xi_2\end{bmatrix}$
is a $\mathbb{Z}_p$-basis for
$(\mathfrak{a}\mathfrak{a}')_p^{-1}$. Furthermore,
$B^{-1}\begin{bmatrix}\xi_1\\\xi_2\end{bmatrix}$
is also a $\mathbb{Z}_p$-basis for  $(\mathfrak{a}\mathfrak{a}')^{-1}_p$ by (\ref{aax}) and (\ref{xBx}). Thus we achieve
\begin{equation}\label{qpsgB}
q_{\xi,\,p}(s_p^{-1})=\gamma_pB^{-1}\quad
\textrm{for some}~\gamma_p\in\mathrm{GL}_2(\mathbb{Z}_p).
\end{equation}
Letting $\gamma=(\gamma_p)\in\prod_p\mathrm{GL}_2(\mathbb{Z}_p)$ we get
\begin{equation}\label{qsgB}
q_\xi(s^{-1})=\gamma B^{-1}.
\end{equation}
We then deduce that
\begin{eqnarray*}
h(C)^{[s,\,K]}&=&(h(\tau)^{\widetilde{A}}|_{\tau=\xi})^{[s,\,K]}\quad
\textrm{by Definition \ref{invariant}}\\
&=&(h(\tau)^{\widetilde{A}})^{\sigma(q_\xi(s^{-1}))}|_{\tau=\xi}
\quad\textrm{by Propositioin \ref{reciprocity}}\\
&=&(h(\tau)^{\widetilde{A}})^{\sigma(\gamma B^{-1})}|_{\tau=\xi}
\quad\textrm{by (\ref{qsgB})}\\
&=&h(\tau)^{\widetilde{A}\widetilde{G}}
|_{\tau=B^{-1}(\xi)}\quad
\textrm{where $G$ is
a matrix in $M_2(\mathbb{Z})$
such that}\\
&&\hspace{3.1cm}G\equiv\gamma_p
\Mod{N M_2(\mathbb{Z}_p)}~\textrm{for all}~p\,|\,N\\
&=&h(\tau)^{\widetilde{A}\widetilde{B}}|_{\tau=\xi''}
\quad\textrm{by (\ref{xBx}) and the fact that for each $p\,|\,N$,}\\
&&\hspace{2.5cm}\textrm{$s_p=1$ and so $\gamma_p B^{-1}=I_2$ owing to (\ref{s1pN}) and (\ref{qpsgB})}\\
&=&h(CC')\quad\textrm{by Definition \ref{invariant} and (\ref{tABx})}.
\end{eqnarray*}
In particular, if we consider the case where $C'=C^{-1}$, then
we derive that
\begin{equation*}
h(C)=h(CC')^{[s^{-1},\,K]}=h(C_0)^{[s^{-1},\,K]}
=h(\tau_K)^{[s^{-1},\,K]}.
\end{equation*}
This implies that $h(C)$ belongs to $K_P$ by Proposition \ref{KPKF}.
\par
For each $p\,\nmid\,N$ and $\mathfrak{p}$ lying above $p$, we
have by (\ref{s1pN}) that $\mathrm{ord}_\mathfrak{p}~
s_p=\mathrm{ord}_\mathfrak{p}~\mathfrak{a}'$, and hence
\begin{equation*}
[s,\,K]|_{K_P}=\sigma(C').
\end{equation*}
Therefore we conclude
\begin{equation*}
h(C)^{\sigma(C')}=h(CC').
\end{equation*}
\end{proof}

\section {Extended form class groups as Galois groups}

With $P$, $K_P$ and $\Gamma$ as in $\S$\ref{classinvariants},
we shall prove our main theorem which asserts that $\mathcal{Q}_N(d_K)/\sim_\Gamma$ can be regarded
as a group isomorphic to $\mathrm{Gal}(K_P/K)$ through the isomorphism described in (\ref{desiredisomorphism}).

\begin{lemma}\label{unit}
If $Q\in\mathcal{Q}(d_K)$ and $\gamma\in I_{\omega_Q}$, then
$j(\gamma,\,\omega_Q)\in\mathcal{O}_K^*$.
\end{lemma}
\begin{proof}
We obtain from $Q=Q^\gamma$ that
\begin{equation*}
[\omega_Q,\,1]=[\gamma(\omega_Q),\,1]
=\frac{1}{j(\gamma,\,\omega_Q)}[\omega_Q,\,1].
\end{equation*}
This claims that $j(\gamma,\,\omega_Q)$ is a unit in $\mathcal{O}_K$.
\end{proof}

\begin{remark}
This lemma can be also justified by using (\ref{isotropy}), (\ref{otherisotropy}) and the property
\begin{equation}\label{jabt}
j(\alpha\beta,\,\tau)=j(\alpha,\,\beta(\tau))j(\beta,\,\tau)
\quad(\alpha,\,\beta\in\mathrm{SL}_2(\mathbb{Z}),\,\tau\in\mathbb{H})
\end{equation}
(\cite[(1.2.4)]{Shimura}).
\end{remark}

\begin{proposition}\label{satisfies}
For given $P$, the group $\Gamma$ satisfies
the property \textup{(\ref{P})}.
\end{proposition}
\begin{proof}
Let $Q=ax^2+bxy+cy^2\in\mathcal{Q}_N(d_K)$ and $\gamma\in\mathrm{SL}_2(\mathbb{Z})$ such that
$Q^{\gamma^{-1}}\in\mathcal{Q}_N(d_K)$.
\par
Assume that $j(\gamma,\,\omega_Q)\mathcal{O}_K\in P$. Then we have
\begin{equation*}
j(\gamma,\,\omega_Q)\mathcal{O}_K=\frac{\nu_1}{\nu_2}
\mathcal{O}_K\quad\textrm{for some}~\nu_1,\,\nu_2\in\mathcal{O}_K-\{0\}
\end{equation*}
satisfying
\begin{equation}\label{ntnt}
\nu_1\equiv t_1,\,\nu_2\equiv t_2\Mod{\mathfrak{n}}~
\textrm{with}~t_1,\,t_2\in T
\end{equation}
and hence
\begin{equation}\label{zjnn}
\zeta j(\gamma,\,\omega_Q)=\frac{\nu_1}{\nu_2}\quad
\textrm{for some}~\zeta\in\mathcal{O}_K^*.
\end{equation}
For convenience, let
$j=j(\gamma,\,\omega_Q)$ and $Q'=Q^{\gamma^{-1}}$.
Then we deduce
\begin{equation}\label{newQQ}
\gamma(\omega_Q)=\omega_{Q'}
\end{equation}
and
\begin{equation*}
[\omega_Q,\,1]=j[\gamma(\omega_Q),\,1]=j[\omega_{Q'},\,1]=
\zeta j[\omega_{Q'},\,1].
\end{equation*}
So there is $\alpha=\begin{bmatrix}r&s\\
u&v\end{bmatrix}\in\mathrm{GL}_2(\mathbb{Z})$ which yields
\begin{equation}\label{zz}
\begin{bmatrix}\zeta j\omega_{Q'}\\
\zeta j\end{bmatrix}=\alpha
\begin{bmatrix}\omega_Q\\1\end{bmatrix}.
\end{equation}
Here, since $\zeta j\omega_{Q'}/\zeta j=\omega_{Q'},\,
\omega_Q\in\mathbb{H}$, we get
$\alpha\in\mathrm{SL}_2(\mathbb{Z})$ and
\begin{equation}\label{newww}
\omega_{Q'}=\alpha(\omega_Q).
\end{equation}
Thus we attain $\gamma(\omega_Q)=\omega_{Q'}=\alpha(\omega_Q)$
by (\ref{newQQ}) and (\ref{newww}), from which we get
$\omega_Q=(\alpha^{-1}\gamma)(\omega_Q)$ and so
\begin{equation}\label{gaI}
\gamma\in\alpha\cdot I_{\omega_Q}.
\end{equation}
Now that $aj\in\mathcal{O}_K$, we
see from (\ref{ntnt}), (\ref{zjnn}) and (\ref{zz}) that
\begin{eqnarray*}
&&a\nu_2(\zeta j)\equiv a\nu_1\equiv at_1\Mod{\mathfrak{n}},~
\textrm{and}\\
&&a\nu_2(\zeta j)\equiv
a\nu_2(u\omega_Q+v)\equiv
ut_2(a\omega_Q)+at_2v\Mod{\mathfrak{n}}.
\end{eqnarray*}
It then follows that
\begin{equation*}
at_1\equiv ut_2(a\omega_Q)+at_2v\Mod{\mathfrak{n}}
\end{equation*}
and hence
\begin{equation*}
ut_2(a\omega_Q)+a(t_2v-t_1)\equiv0\Mod{\mathfrak{n}}.
\end{equation*}
Since $\mathfrak{n}=N\mathcal{O}_K=N[a\omega_Q,\,1]$, we have
\begin{equation*}
ut_2\equiv0\Mod{N}\quad\textrm{and}\quad
a(t_2v-t_1)\equiv0\Mod{N}.
\end{equation*}
Moreover, since $\gcd(N,\,t_1)=\gcd(N,\,t_2)=\gcd(N,\,a)=1$, we achieve that
\begin{equation*}
u\equiv0\Mod{N}\quad\textrm{and}\quad
v\equiv t_1t_2^{-1}\Mod{N}
\end{equation*}
where $t_2^{-1}$ is an integer satisfying $t_2t_2^{-1}\equiv1\Mod{N}$. This, together with
the facts $\det(\alpha)=1$ and $T$ is a subgroup of
$(\mathbb{Z}/N\mathbb{Z})^*$, implies
$\alpha=\begin{bmatrix}r&s\\u&v\end{bmatrix}\in\Gamma$.
Therefore we conclude $\gamma\in \Gamma\cdot I_{\omega_Q}$
by (\ref{gaI}), as desired.
\par
Conversely, assume that
$\gamma\in\Gamma\cdot I_{\omega_Q}$, and so
\begin{equation*}
\gamma=\alpha\beta\quad
\textrm{for some}~\alpha=
\begin{bmatrix}r&s\\u&v\end{bmatrix}\in\Gamma~
\textrm{and}~\beta\in I_{\omega_Q}.
\end{equation*}
Here we observe that
\begin{equation}\label{u0vt}
u\equiv0\Mod{N}\quad\textrm{and}\quad
v\equiv t\Mod{N}~\textrm{for some}~t\in T.
\end{equation}
We then derive that
\begin{eqnarray*}
j(\gamma,\,\omega_Q)&=&j(\alpha\beta,\,\omega_Q)\\
&=&j(\alpha,\,\beta(\omega_Q))j(\beta,\,\omega_Q)\quad
\textrm{by (\ref{jabt})}\\
&=&j(\alpha,\,\omega_Q)\zeta\quad\textrm{for some}~
\zeta\in\mathcal{O}_K^*~\textrm{by the fact
$\beta\in I_{\omega_Q}$ and Lemma \ref{unit}}.
\end{eqnarray*}
Thus we attain
\begin{equation*}
\zeta^{-1}j(\gamma,\,\omega_Q)-v
=j(\alpha,\,\omega_Q)-v
=(u\omega_Q+v)-v
=\frac{1}{a}\{u(a\omega_Q)\}.
\end{equation*}
And, it follows from the fact $\gcd(N,\,a)=1$ and (\ref{u0vt}) that
\begin{equation*}
\zeta^{-1}j(\gamma,\,\omega_Q)\equiv^*v
\equiv^*t\Mod{\mathfrak{n}}.
\end{equation*}
This shows that $\zeta^{-1}j(\gamma,\,\omega_Q)
\mathcal{O}_K\in P$, and hence $j(\gamma,\,\omega_Q)\mathcal{O}_K\in P$.
\par
Therefore, the group $\Gamma$ satisfies the property (\ref{P}) for $P$.
\end{proof}

\begin{theorem}\label{Galoisgroups}
We have an isomorphism
\begin{equation}\label{Galoisisomorphism}
\begin{array}{ccl}
\mathcal{Q}_N(d_K)/\sim_\Gamma&\rightarrow&
\mathrm{Gal}(K_P/K)\\
\left[Q\right]&\mapsto&
\left(h(\tau_K)\mapsto h(-\overline{\omega}_Q)~|~
h\in\mathcal{F}_{\Gamma,\,\mathbb{Q}}~
\textrm{is finite at}~\tau_K\right).
\end{array}
\end{equation}
\end{theorem}
\begin{proof}
By Theorem \ref{formclassgroup} and Proposition \ref{satisfies}
one may consider
$\mathcal{Q}_N(d_K)/\sim_\Gamma$ as a group isomorphic to
$I_K(\mathfrak{n})/P$ via the isomorphism $\phi_\Gamma$ in $\S$\ref{sect2}.
Let $C\in\mathrm{Cl}(P)$ and so
\begin{equation*}
C=\phi_\Gamma([Q])=[[\omega_Q,\,1]]
\quad\textrm{for some}~Q\in\mathcal{Q}_N(d_K)/\sim_\Gamma.
\end{equation*}
Note that
$C$ contains an integral ideal
$\mathfrak{a}=a^{\varphi(N)}[\omega_Q,\,1]$, where
$\varphi$ is the Euler totient function. We establish
by Lemma \ref{prime} and the definition (\ref{wQ}) that
\begin{equation*}
\mathfrak{a}^{-1}=\frac{1}{\mathrm{N}_{K/\mathbb{Q}}(\mathfrak{a})}
\overline{\mathfrak{a}}=\frac{1}{a^{\varphi(N)-1}}
[-\overline{\omega}_Q,\,1]
\end{equation*}
and
\begin{equation*}
\begin{bmatrix}\tau_K\\1\end{bmatrix}=
\begin{bmatrix}
a^{\varphi(N)} & -a^{\varphi(N)-1}(b+b_K)/2\\
0&a^{\varphi(N)-1}
\end{bmatrix}
\begin{bmatrix}
-\overline{\omega}_Q/a^{\varphi(N)-1}\\
1/a^{\varphi(N)-1}
\end{bmatrix}
\end{equation*}
where $\min(\tau_K,\,\mathbb{Q})=x^2+b_Kx+c_K$ ($\in\mathbb{Z}[x]$).
We then derive by Proposition \ref{Galoisdecomposition} that
if $h\in\mathcal{F}_{\Gamma,\,\mathbb{Q}}$ is finite at $\tau_K$, then
\begin{eqnarray*}
h(C)&=&h(\tau)^{\widetilde{\left[\begin{smallmatrix}
a^{\varphi(N)}&-a^{\varphi(N)-1}(b+b_K)/2\\
0&a^{\varphi(N)-1}
\end{smallmatrix}\right]}}|_{\tau=-\overline{\omega}_Q}
\quad\textrm{by Definition \ref{invariant}}\\
&&\hspace{5cm}\textrm{where $\widetilde{~\cdot~}$ means the reduction
onto $\mathrm{GL}_2(\mathbb{Z}/N\mathbb{Z})/\{\pm I_2\}$}\\
&=&h(\tau)^{\widetilde{\left[\begin{smallmatrix}
1 & -a^{-1}(b+b_K)/2\\
0&a^{-1}
\end{smallmatrix}\right]}}|_{\tau=-\overline{\omega}_Q}
\quad\textrm{since}~a^{\varphi(N)}\equiv1\Mod{N}\\
&&\hspace{5cm}\textrm{where $a^{-1}$ is an integer such that
$aa^{-1}\equiv1\Mod{N}$}\\
&=&h(\tau)^{\widetilde{\left[\begin{smallmatrix}
1&0\\0&a^{-1}
\end{smallmatrix}\right]}
\widetilde{\left[\begin{smallmatrix}
1 &-a^{-1}(b+b_K)/2\\0&1
\end{smallmatrix}\right]}}|_{\tau=-\overline{\omega}_Q}\\
&=&h(\tau)^{\widetilde{\left[\begin{smallmatrix}
1 &-a^{-1}(b+b_K)/2\\0&1
\end{smallmatrix}\right]}}|_{\tau=-\overline{\omega}_Q}
\quad\textrm{because $h(\tau)$ has rational Fourier coefficients}\\
&=&h(\tau)|_{\tau=-\overline{\omega}_Q}
\quad\textrm{since $h(\tau)$ is modular for $\Gamma$}\\
&=&h(-\overline{\omega}_Q).
\end{eqnarray*}
Now, the isomorphism $\phi_\Gamma$ followed by
the isomorphism
\begin{equation*}
\begin{array}{ccl}
\mathrm{Cl}(P)&\rightarrow&\mathrm{Gal}(K_P/K)\\
C&\mapsto&\left(
h(\tau_K)=h(C_0)\mapsto
h(C_0)^{\sigma(C)}=h(C)=h(-\overline{\omega}_Q)~|~
h\in\mathcal{F}_{\Gamma,\,\mathbb{Q}}~
\textrm{is finite at}~\tau_K\right)
\end{array}
\end{equation*}
which is induced from Propositions \ref{KPKF}, \ref{transformation}
and Remark \ref{identityinvariant}, yields the isomorphism stated in (\ref{Galoisisomorphism}),
as desired.
\end{proof}

\begin{remark}\label{difference}
In \cite{E-K-S17} Eum, Koo and Shin
considered only the case
where $K\neq\mathbb{Q}(\sqrt{-1}),\,\mathbb{Q}(\sqrt{-3})$,
$P=P_{K,\,1}(\mathfrak{n})$ and $\Gamma=\Gamma_1(N)$.
As for the group operation of $\mathcal{Q}_N(d_K)/\sim_{\Gamma_1(N)}$
one can refer to \cite[Remark 2.10]{E-K-S17}.
They
established an isomorphism
\begin{equation}\label{rayGaloisgroups}
\begin{array}{lcl}
\mathcal{Q}_N(d_K)/\sim_{\Gamma_1(N)}&\rightarrow&\mathrm{Gal}(K_\mathfrak{n}/K)\\
\left[Q\right]=\left[ax^2+bxy+cy^2\right]&\mapsto&
\left(
h(\tau_K)\mapsto h(\tau)^{\widetilde{\left[
\begin{smallmatrix}a&(b-b_K)/2\\0&1\end{smallmatrix}\right]}}|_{\tau=
\omega_Q}~|~h(\tau)\in\mathcal{F}_N~
\textrm{is finite at $\tau_K$}\right).
\end{array}
\end{equation}
The difference between the isomorphisms described in (\ref{Galoisisomorphism}) and (\ref{rayGaloisgroups}) arises from the Definition \ref{invariant}
of $h(C)$. The invariant $h_\mathfrak{n}(C)$ appeared in \cite[Definition 3.3]{E-K-S17} coincides with $h(C^{-1})$.
\end{remark}

\section {Finding representatives of extended form classes}

In this last section, by improving the proof of
Proposition \ref{surjective} further, we shall explain how to find
all quadratic forms which represent distinct classes in
$\mathcal{Q}_N(d_K)/\sim_\Gamma$.
\par
For a given $Q=ax^2+bxy+cy^2\in\mathcal{Q}_N(d_K)$ we define an equivalence
relation $\equiv_Q$ on $M_{1,\,2}(\mathbb{Z})$ as follows:
Let $\begin{bmatrix}r&s\end{bmatrix},\,
\begin{bmatrix}u&v\end{bmatrix}\in M_{1,\,2}(\mathbb{Z})$. Then,
$\begin{bmatrix}r&s\end{bmatrix}\equiv_Q\begin{bmatrix}u&v\end{bmatrix}$
if and only if
\begin{equation*}
\begin{bmatrix}r&s\end{bmatrix}
\equiv\pm t\begin{bmatrix}u&v\end{bmatrix}\gamma\Mod{NM_{1,\,2}(\mathbb{Z})}
\quad\textrm{for some}~t\in T~\textrm{and}~\gamma\in
\Gamma_Q
\end{equation*}
where
\begin{equation*}
\Gamma_Q=\left\{\begin{array}{ll}
\left\{\pm I_2\right\} & \textrm{if}~d_K\neq-4,\,-3,\vspace{0.1cm}\\
\left\{\pm I_2,\,
\pm\begin{bmatrix}
-b/2&-a^{-1}(b^2+4)/4)\\a&b/2
\end{bmatrix}
\right\} & \textrm{if}~d_K=-4,\vspace{0.1cm}\\
\left\{\pm I_2,\,
\pm\begin{bmatrix}
-(b+1)/2&-a^{-1}(b^2+3)/4\\a&(b-1)/2
\end{bmatrix},\,
\pm\begin{bmatrix}
(b-1)/2&a^{-1}(b^2+3)/4\\-a&-(b+1)/2
\end{bmatrix}
\right\} & \textrm{if}~d_K=-3.
\end{array}
\right.
\end{equation*}
Here, $a^{-1}$ is an integer satisfying $aa^{-1}\equiv1\Mod{N}$.

\begin{lemma}\label{equivGamma}
Let $Q=ax^2+bxy+cy^2\in\mathcal{Q}_N(d_K)$ and
$\begin{bmatrix}r&s\end{bmatrix},\,
\begin{bmatrix}u&v\end{bmatrix}\in M_{1,\,2}(\mathbb{Z})$ such that
$\gcd(N,\,Q(s,\,-r))=\gcd(N,\,Q(v,\,-u))=1$.
Then,
\begin{equation*}
[(r\omega_Q+s)\mathcal{O}_K]=[(u\omega_Q+v)\mathcal{O}_K]~
\textrm{in}~P_K(\mathfrak{n})/P
\quad
\Longleftrightarrow\quad
\begin{bmatrix}r&s\end{bmatrix}\equiv_Q
\begin{bmatrix}u&v\end{bmatrix}.
\end{equation*}
\end{lemma}
\begin{proof}
Note that by Lemma \ref{prime} (i) the fractional ideals $(r\omega_Q+s)\mathcal{O}_K$ and
$(u\omega_Q+v)\mathcal{O}_K$ belong to $P_K(\mathfrak{n})$.
Furthermore, we know that
\begin{equation}\label{OK*}
\mathcal{O}_K^*=
\left\{
\begin{array}{ll}
\{\pm1\} & \textrm{if}~K\neq\mathbb{Q}(\sqrt{-1}),\,\mathbb{Q}(\sqrt{-3}),\\
\{\pm 1,\,\pm\tau_K\} & \textrm{if}~K=\mathbb{Q}(\sqrt{-1}),\\
\{\pm 1,\,\pm\tau_K,\,
\pm\tau_K^2\} & \textrm{if}~K=\mathbb{Q}(\sqrt{-3})
\end{array}\right.
\end{equation}
(\cite[Exercise 5.9]{Cox}) and so
\begin{equation}\label{UK}
U_K=\{(m,\,n)\in\mathbb{Z}^2~|~m\tau_K+n\in\mathcal{O}_K^*\}
=\left\{\begin{array}{ll}
\{\pm(0,\,1)\} & \textrm{if}~K\neq\mathbb{Q}(\sqrt{-1}),\,\mathbb{Q}(\sqrt{-3}),\\
\{\pm(0,\,1),\,\pm(1,\,0)\} & \textrm{if}~K=\mathbb{Q}(\sqrt{-1}),\\
\{\pm(0,\,1),\,\pm(1,\,0),\,\pm(1,\,1)\} & \textrm{if}~K=\mathbb{Q}(\sqrt{-3}).
\end{array}\right.
\end{equation}
Then we achieve that
\begin{eqnarray*}
&&[(r\omega_Q+s)\mathcal{O}_K]=[(u\omega_Q+v)\mathcal{O}_K]
\quad\textrm{in}~P_K(\mathfrak{n})/P\\
&\Longleftrightarrow&
\left(\frac{r\omega_Q+s}{u\omega_Q+v}\right)\mathcal{O}_K\in P\\
&\Longleftrightarrow&
\frac{r\omega_Q+s}{u\omega_Q+v}\equiv^*\zeta t\Mod{\mathfrak{n}}
\quad\textrm{for some}~\zeta\in\mathcal{O}_K^*~\textrm{and}~t\in T\\
&\Longleftrightarrow&
a(r\omega_Q+s)\equiv\zeta ta(u\omega_Q+v)\Mod{\mathfrak{n}}
\quad\textrm{since $a(u\omega_Q+v)\mathcal{O}_K$
is relatively prime to $\mathfrak{n}$}\\
&&\hspace{6.5cm}\textrm{and}~a\omega_Q\in\mathcal{O}_K\\
&\Longleftrightarrow&
r\left(\tau_K+\frac{b_K-b}{2}\right)+as\equiv
(m\tau_K+n)t\left\{
u\left(\tau_K+\frac{b_K-b}{2}\right)+av\right\}\Mod{\mathfrak{n}}\\
&&\hspace{10.5cm}\textrm{for some}~(m,\,n)\in U_K\\
&\Longleftrightarrow&
r\tau_K+\left(\frac{r(b_K-b)}{2}+as\right)\equiv
t(-mub_K+mk+nu)\tau_K+t(-muc_K+nk)
\Mod{\mathfrak{n}}\\
&&\hspace{3.5cm}\textrm{with}~
k=\frac{u(b_K-b)}{2}+av,~\textrm{where}~\min(\tau_K,\,\mathbb{Q})
=x^2+b_Kx+c_K\\
&\Longleftrightarrow&
r\equiv t\left\{-\left(\frac{b_K+b}{2}\right)m+n\right\}u+tmav\Mod{N}\quad\textrm{and}\\
&&s\equiv ta^{-1}\left(\frac{b_K^2-b^2}{4}-c_K\right)mu+
t\left\{-\left(\frac{b_K-b}{2}\right)m+n\right\}v\Mod{N}\\
&&\hspace{10cm}\textrm{by the fact}~\mathfrak{n}=N[\tau_K,\,1]\\
&\Longleftrightarrow&
\begin{bmatrix}r&s\end{bmatrix}\equiv_Q
\begin{bmatrix}u&v\end{bmatrix}\quad\textrm{by (\ref{UK}) and
the definition of $\equiv_Q$}.
\end{eqnarray*}
\end{proof}

For each $Q\in\mathcal{Q}_N(d_K)$, let
\begin{equation*}
M_Q=\left\{\begin{bmatrix}u&v\end{bmatrix}\in M_{1,\,2}(\mathbb{Z})~|~
\gcd(N,\,Q(v,\,-u))=1\right\}.
\end{equation*}

\begin{proposition}\label{algorithm}
One can explicitly find quadratic forms
representing all distinct classes in $\mathcal{Q}_N(d_K)/\sim_\Gamma$.
\end{proposition}
\begin{proof}
We adopt the idea in the proof of Proposition \ref{surjective}.
Let $Q_1',\,Q_2',\,\ldots,\,Q_h'$
be quadratic forms in $\mathcal{Q}_N(d_K)$ which represent all distinct classes in
$\mathrm{C}(d_K)=\mathcal{Q}(d_K)/\sim$. Then we get by Lemma \ref{equivGamma} that for each $i=1,\,2,\,\ldots,\,h$
\begin{equation*}
P_K(\mathfrak{n})/P=
\left\{[(u\omega_{Q_i'}+v)\mathcal{O}_K]~|~
\begin{bmatrix}u&v\end{bmatrix}\in M_{Q_i'}/\equiv_{Q_i'}\right\}
=\left\{\left[\frac{1}{u\omega_{Q_i'}+v}\mathcal{O}_K\right]~|~
\begin{bmatrix}u&v\end{bmatrix}\in M_{Q_i'}/\equiv_{Q_i'}\right\}.
\end{equation*}
Thus we obtain by (\ref{decomp}) that
\begin{eqnarray*}
I_K(\mathfrak{n})/P&=&(P_K(\mathfrak{n})/P)\cdot
\{[[\omega_{Q_i'},\,1]]\in I_K(\mathfrak{n})/P~|~i=1,\,2,\,\ldots,\,h\}\\
&=&\left\{\left[\frac{1}{u\omega_{Q_i'}+v}
[\omega_{Q_i'},\,1]\right]~|~i=1,\,2,\,\ldots,\,h~
\textrm{and}~\begin{bmatrix}u&v\end{bmatrix}
\in M_{Q_i'}/\equiv_{Q_i'}\right\}\\
&=&\left\{
\left[\left[\begin{bmatrix}\mathrm{*}&\mathrm{*}\\
\widetilde{u}&\widetilde{v}\end{bmatrix}(\omega_{Q_i'}),\,1\right]\right]~|~
i=1,\,2,\,\ldots,\,h~
\textrm{and}~\begin{bmatrix}u&v\end{bmatrix}
\in M_{Q_i'}/\equiv_{Q_i'}\right\}
\end{eqnarray*}
where $\begin{bmatrix}\mathrm{*}&\mathrm{*}\\
\widetilde{u}&\widetilde{v}\end{bmatrix}$ is a matrix in
$\mathrm{SL}_2(\mathbb{Z})$ such that
$\begin{bmatrix}\mathrm{*}&\mathrm{*}\\
\widetilde{u}&\widetilde{v}\end{bmatrix}\equiv
\begin{bmatrix}\mathrm{*}&\mathrm{*}\\u&v\end{bmatrix}
\Mod{NM_2(\mathbb{Z})}$.
Therefore we conclude
\begin{equation*}
\mathcal{Q}_N(d_K)/\sim_\Gamma=\left\{
\left[Q_i'^{\left[\begin{smallmatrix}
\mathrm{*}&\mathrm{*}\\\widetilde{u}&\widetilde{v}
\end{smallmatrix}\right]^{-1}}
\right]~|~i=1,\,2,\,\ldots,\,h~
\textrm{and}~\begin{bmatrix}u&v\end{bmatrix}
\in M_{Q_i'}/\equiv_{Q_i'}
\right\}.
\end{equation*}
\end{proof}

\begin{example}\label{example}
Let $K=\mathbb{Q}(\sqrt{-5})$,
$N=12$ and $T=(\mathbb{Z}/N\mathbb{Z})^*$. Then
we get $P=P_{K,\,\mathbb{Z}}(\mathfrak{n})$
and $K_P=H_\mathcal{O}$, where $\mathfrak{n}=N\mathcal{O}_K$ and $\mathcal{O}$ is the order of conductor $N$ in $K$. There are two reduced forms of discriminant $d_K=-20$, namely
\begin{equation*}
Q_1=x^2+5y^2\quad\textrm{and}\quad
Q_2=2x^2+2xy+3y^2.
\end{equation*}
Set
\begin{equation*}
Q_1'=Q_1\quad
\textrm{and}\quad
Q_2'=Q_2^{\left[\begin{smallmatrix}1&1\\1&2\end{smallmatrix}\right]}
=7x^2+22xy+18y^2
\end{equation*}
which belong to $\mathcal{Q}_N(d_K)$.
We then see that
\begin{equation*}
M_{Q_1'}/\equiv_{Q_1'}=\left\{
\begin{bmatrix}0&1\end{bmatrix},\,
\begin{bmatrix}1&0\end{bmatrix},\,
\begin{bmatrix}1&6\end{bmatrix},\,
\begin{bmatrix}2&3\end{bmatrix},\,
\begin{bmatrix}3&2\end{bmatrix},\,
\begin{bmatrix}3&4\end{bmatrix},\,
\begin{bmatrix}4&3\end{bmatrix},\,
\begin{bmatrix}6&1\end{bmatrix}
\right\}
\end{equation*}
with corresponding matrices
\begin{eqnarray*}
\begin{bmatrix}
1&0\\0&1
\end{bmatrix},\,
\begin{bmatrix}
0&-1\\1&0
\end{bmatrix},\,
\begin{bmatrix}
0&-1\\1&6
\end{bmatrix},\,
\begin{bmatrix}
1&1\\2&3
\end{bmatrix},\,
\begin{bmatrix}
-1&-1\\3&2
\end{bmatrix},\,
\begin{bmatrix}
1&1\\3&4
\end{bmatrix},\,
\begin{bmatrix}
-1&-1\\4&3
\end{bmatrix},\,
\begin{bmatrix}
1&0\\6&1
\end{bmatrix}
\end{eqnarray*}
and
\begin{equation*}
M_{Q_2'}/\equiv_{Q_2'}=\left\{
\begin{bmatrix}0&1\end{bmatrix},\,
\begin{bmatrix}1&5\end{bmatrix},\,
\begin{bmatrix}1&11\end{bmatrix},\,
\begin{bmatrix}2&1\end{bmatrix},\,
\begin{bmatrix}3&1\end{bmatrix},\,
\begin{bmatrix}3&7\end{bmatrix},\,
\begin{bmatrix}4&5\end{bmatrix},\,
\begin{bmatrix}6&1\end{bmatrix}
\right\}
\end{equation*}
with corresponding matrices
\begin{eqnarray*}
\begin{bmatrix}
1&0\\0&1
\end{bmatrix},\,
\begin{bmatrix}
0&-1\\1&5
\end{bmatrix},\,
\begin{bmatrix}
0&-1\\1&11
\end{bmatrix},\,
\begin{bmatrix}
-1&-1\\2&1
\end{bmatrix},\,
\begin{bmatrix}
1&0\\3&1
\end{bmatrix},\,
\begin{bmatrix}
1&2\\3&7
\end{bmatrix},\,
\begin{bmatrix}
1&1\\4&5
\end{bmatrix},\,
\begin{bmatrix}
1&0\\6&1
\end{bmatrix}.
\end{eqnarray*}
Hence there are $16$ quadratic forms
\begin{equation*}
\begin{array}{llll}
x^2+5y^2,&5x^2+y^2,&41x^2+12xy+y^2,&29x^2-26xy+6y^2,\\ 49x^2+34xy+6y^2,&61x^2-38xy+6y^2,&89x^2+46xy+6y^2,&181x^2-60xy+5y^2,\\
7x^2+22xy+18y^2,&83x^2+48xy+7y^2,&623x^2+132xy+7y^2,&35x^2+20xy+3y^2,\\
103x^2-86xy+18y^2,&43x^2-18xy+2y^2,&23x^2-16xy+3y^2,&523x^2-194xy+18y^2
\end{array}
\end{equation*}
which represent all distinct classes in $\mathcal{Q}_N(d_K)/
\sim_\Gamma=\mathcal{Q}_{12}(-20)/\sim_{\Gamma_0(12)}$.
\par
On the other hand, for
$\begin{bmatrix}r_1&r_2\end{bmatrix}\in M_{1,\,2}(\mathbb{Q})
\setminus M_{1,\,2}(\mathbb{Z})$ the \textit{Siegel function}
$g_{\left[\begin{smallmatrix}r_1&r_2\end{smallmatrix}\right]}(\tau)$
is given by the infinite product
\begin{eqnarray*}
g_{\left[\begin{smallmatrix}r_1&r_2\end{smallmatrix}\right]}(\tau)
&=&-e^{\pi\mathrm{i}r_2(r_1-1)}
q^{(1/2)(r_1^2-r_1+1/6)}
(1-q^{r_1}e^{2\pi\mathrm{i}r_2})\\
&&\times
\prod_{n=1}^\infty(1-q^{n+r_1}e^{2\pi\mathrm{i}r_2})
(1-q^{n-r_1}e^{-2\pi\mathrm{i}r_2})\quad(\tau\in\mathbb{H})
\end{eqnarray*}
which generalizes the Dedekind eta-function
$\displaystyle q^{1/24}\prod_{n=1}^\infty
(1-q^n)$.
Then the function
\begin{equation*}
g_{\left[\begin{smallmatrix}
1/2&0\end{smallmatrix}\right]}(12\tau)^{12}
=\left(\frac{\eta(6\tau)}{\eta(12\tau)}\right)^{24}
\end{equation*}
belongs to
$\mathcal{F}_{\Gamma_0(12),\,\mathbb{Q}}$ (\cite[Theorem 1.64]{Ono} or \cite{K-L}), and the Galois conjugates of $g_{\left[\begin{smallmatrix}
1/2&0\end{smallmatrix}\right]}(12\tau_K)^{12}$ over $K$
are
\begin{equation*}
\begin{array}{llll}
g_1=g_{\left[\begin{smallmatrix}
1/2&0\end{smallmatrix}\right]}(12\sqrt{-5})^{12}, &
g_2=g_{\left[\begin{smallmatrix}
1/2&0\end{smallmatrix}\right]}(12\sqrt{-5}/5)^{12}, \\ g_3=g_{\left[\begin{smallmatrix}
1/2&0\end{smallmatrix}\right]}(12(6+\sqrt{-5})/41)^{12}, & g_4=g_{\left[\begin{smallmatrix}
1/2&0\end{smallmatrix}\right]}(12(-13+\sqrt{-5})/29)^{12},\\
g_5=g_{\left[\begin{smallmatrix}
1/2&0\end{smallmatrix}\right]}(12(17+\sqrt{-5})/49)^{12}, &
g_6=g_{\left[\begin{smallmatrix}
1/2&0\end{smallmatrix}\right]}(12(-19+\sqrt{-5})/61)^{12}, \\ g_7=g_{\left[\begin{smallmatrix}
1/2&0\end{smallmatrix}\right]}(12(23+\sqrt{-5})/89)^{12}, & g_8=g_{\left[\begin{smallmatrix}
1/2&0\end{smallmatrix}\right]}(12(-30+\sqrt{-5})/181)^{12},\\
g_9=g_{\left[\begin{smallmatrix}
1/2&0\end{smallmatrix}\right]}(12(11+\sqrt{-5})/7)^{12}, &
g_{10}=g_{\left[\begin{smallmatrix}
1/2&0\end{smallmatrix}\right]}(12(24+\sqrt{-5})/83)^{12}, \\ g_{11}=g_{\left[\begin{smallmatrix}
1/2&0\end{smallmatrix}\right]}(12(66+\sqrt{-5})/623)^{12}, & g_{12}=g_{\left[\begin{smallmatrix}
1/2&0\end{smallmatrix}\right]}(12(10+\sqrt{-5})/35)^{12},\\
g_{13}=g_{\left[\begin{smallmatrix}
1/2&0\end{smallmatrix}\right]}(12(-43+\sqrt{-5})/103)^{12}, &
g_{14}=g_{\left[\begin{smallmatrix}
1/2&0\end{smallmatrix}\right]}(12(-9+\sqrt{-5})/43)^{12}, \\ g_{15}=g_{\left[\begin{smallmatrix}
1/2&0\end{smallmatrix}\right]}(12(-8+\sqrt{-5})/23)^{12}, & g_{16}=g_{\left[\begin{smallmatrix}
1/2&0\end{smallmatrix}\right]}(12(-97+\sqrt{-5})/523)^{12}
\end{array}
\end{equation*}
possibly with some multiplicity. Now, we evaluate
\begin{eqnarray*}
&&\prod_{i=1}^{16}(x-g_i)\\&=&
x^{16}+1251968x^{15}-14929949056x^{14}+1684515904384x^{13}
-61912544374756x^{12}\\
&&+362333829428160x^{11}
+32778846351721632x^{10}-845856631699319872x^9\\
&&+4605865492693542918x^8+91164259067285621248x^7
-124917935291699694528x^6\\
&&+180920285564131280640x^5
-3000295144057714916x^4+8871452719720384x^3\\
&&+458008762175904x^2
-1597177179712x+1
\end{eqnarray*}
with nonzero discriminant. Thus
$g_{\left[\begin{smallmatrix}
1/2&0\end{smallmatrix}\right]}(12\tau_K)^{12}$ generates $K_P=H_\mathcal{O}$ over $K$.
\end{example}

\begin{remark}
In \cite{Schertz} Schertz
deals with various constructive problems on
the theory of complex multiplication
in terms of the Dedekind eta-function and Siegel function. See also \cite{K-L} and \cite{Ramachandra}.
\end{remark}

\bibliographystyle{amsplain}

\address{
Applied Algebra and Optimization Research Center\\
Sungkyunkwan University\\
Suwon-si, Gyeonggi-do 16419\\
Republic of Korea} {hoyunjung@skku.edu}
\address{
Department of Mathematical Sciences \\
KAIST \\
Daejeon 34141\\
Republic of Korea} {jkkoo@math.kaist.ac.kr}
\address{
Department of Mathematics\\
Hankuk University of Foreign Studies\\
Yongin-si, Gyeonggi-do 17035\\
Republic of Korea} {dhshin@hufs.ac.kr}

\end{document}